\documentclass{amsart}

\input epsf%
\usepackage{multirow}
   \long\def\comment#1{}
\makeatletter
\newcommand*{\rom}[1]{\expandafter\@slowromancap\romannumeral #1@}
\makeatother
\usepackage{latexsym}
\usepackage{rotating}
\usepackage{xcolor,pict2e}
\usepackage{amssymb,amsmath,amscd}
\usepackage[ruled,vlined, boxed]{algorithm2e}
\usepackage{bm}\usepackage{txfonts}
\usepackage{stmaryrd}
\usepackage{amsfonts}
\usepackage{mathtools}
\usepackage{epstopdf}
\ifpdf
  \DeclareGraphicsExtensions{.eps,.pdf,.png,.jpg}
\else
  \DeclareGraphicsExtensions{.eps}
\fi

\usepackage{graphicx}
\usepackage{subfigure}
\newtheorem{theorem}{Theorem}[section]
\newtheorem{lemma}[theorem]{Lemma}

\newtheorem{remark}[theorem]{Remark}

\theoremstyle{definition}
\def\ad#1{\begin{aligned}#1\end{aligned}}  
 \def\an#1{\begin{align}#1\end{align}}
 
\def\p#1{\begin{pmatrix}#1\end{pmatrix}} 
  \numberwithin{equation}{section}
\numberwithin{table}{section}
\numberwithin{figure}{section}
\def\div{\operatorname{div}}

\def\boxit#1{\vbox{\hrule height1pt \hbox{\vrule width1pt\kern1pt
     #1\kern1pt\vrule width1pt}\hrule height1pt }}
 \def\lab#1{\boxit{\small #1}\label{#1}}

  \def\lab#1{\label{#1}}  
\newcommand{\bc}{\begin{center}}
\newcommand{\ec}{\end{center}}
\newcommand{\be}{\begin{eqnarray}}
\newcommand{\ee}{\end{eqnarray}}

\newcommand{\ben}{\begin{eqnarray*}}
\newcommand{\een}{\end{eqnarray*}}
\newcommand{\bS}{\mathbb{S}}

\newcommand{\inteu}{\hat{u}}

\newcommand{\TT}{\mathbb{T}}

\newcommand{\ddiv}{{\rm div}}

\def\cE{\mathcal{E}}

\def\cT{\mathcal{T}}
\def\R{\mathbb{R}}

\newcommand{\disp}{\displaystyle}
\newcommand\bx{\boldsymbol{x}}

\newcommand{\T}{\mathcal{T}}
\newcommand{\bq}{\begin{equation}}
\newcommand{\eq}{\end{equation}}
\long\def\comments#1{ #1}
\newcommand{\tangevec}{t}
\newcommand{\normalvec}{n}

\begin{document}

\comments{ }

\title[]{
Partial relaxation of  $C^0$ vertex continuity of stresses of conforming mixed finite elements for the elasticity problem
}


\author {Jun Hu}
\address{LMAM and School of Mathematical Sciences, Peking University,
  Beijing 100871, P. R. China.  hujun@math.pku.edu.cn}

\author{Rui Ma}
\address{LMAM and School of Mathematical Sciences, Peking University,
  Beijing 100871, P. R. China. maruipku@gmail.com}
\thanks{The authors were supported by  NSFC
projects 11625101 and 11421101}

\begin{abstract}

  \vskip 15pt
 A conforming triangular mixed  element recently proposed by Hu and Zhang for linear elasticity   is extended  by rearranging the global degrees of freedom. More precisely, adaptive meshes   $\mathcal{T}_1$, $\cdots$, $\mathcal{T}_N$ which  are successively  refined from an initial mesh $\mathcal{T}_0$ through a newest vertex bisection strategy,  admit a crucial hierarchical structure, namely, a newly added vertex $\bx_e$ of the mesh $\mathcal{T}_\ell$  is the midpoint of an edge $e$ of the coarse mesh $\mathcal{T}_{\ell-1}$. Such a hierarchical structure is explored to partially relax the $C^0$ vertex continuity of symmetric matrix-valued functions in the discrete stress space of the original element on $\mathcal{T}_\ell$ and results in an extended discrete stress space: for such an internal vertex $\bx_e$ located at the coarse edge $e$ with the unit tangential vector $t_e$ and the unit normal vector $n_e=t_e^\perp$,  the pure tangential component basis function $\varphi_{\bx_e}(\bx)t_e t_e^T$ of the original discrete stress space associated to vertex $\bx_e$ is split into two basis functions $\varphi_{\bx_e}^+(\bx)t_e t_e^T$ and $\varphi_{\bx_e}^-(\bx)t_e t_e^T$ along edge $e$, where $\varphi_{\bx_e}(\bx)$ is the  nodal basis function of the scalar-valued Lagrange element of order $k$ ($k$ is equal to the polynomial degree of the discrete stress)  on  $\mathcal{T}_\ell$ with  $\varphi_{\bx_e}^+(\bx)$ and $\varphi_{\bx_e}^-(\bx)$ denoted its two restrictions on two sides of $e$, respectively. Since the remaining two basis functions $\varphi_{\bx_e}(\bx)n_en_e^T$, $\varphi_{\bx_e}(\bx)(n_et_e^T+t_en_e^T)$  are the same as those associated to $\bx_e$ of the original discrete stress space, the number of the global basis functions  associated to $\bx_e$ of the extended discrete stress space becomes four rather than three (for the original discrete stress space).
 As a result,  though the extended discrete stress space on $\mathcal{T}_\ell$ is still a $H(\div)$  subspace,  the pure tangential component along the coarse edge $e$ of discrete stresses in it is not necessarily continuous at such vertices like $\bx_e$. A feature of this extended discrete stress space  is  its nestedness in the sense that a space on a coarse mesh $\T$ is a subspace of a space on any refinement   $\hat{\T}$ of $\T$, which allows a proof of convergence of a standard adaptive algorithm. The idea is extended to impose a general traction boundary condition on the discrete level. Numerical experiments are provided to illustrate performance on both uniform and adaptive meshes.

  \vskip 15pt

\noindent{\bf Keywords. }{
    linear elasticity, nested mixed finite element, adaptive algorithm}

 \vskip 15pt

\noindent{\bf AMS subject classifications.}
    { 65N30, 74B05.}
\end{abstract}
\maketitle

\section{Introduction}\label{sec:intro}
The problems that are most frequently solved in scientific and engineering computing may probably
be elasticity equations. The finite element method (FEM) was invented in analyzing the stress
of elastic structures in the 1950s. The mixed FEM within the Hellinger-Reissner (H-R) principle
for elasticity yields a direct stress approximation since it takes both the stress and displacement as an
independent variable; while the displacement FEM only gives an indirect stress approximation. Hu,
a founder of the celebrated Hu-Washizu principle for elasticity, pointed out that, the H-R principle is
more general than both  minimum potential and complementary energy principles, and is much
fitter for numerical solutions \cite{Hu1954}. Indeed, the mixed FEM can be free of locking
for nearly incompressible materials, and be applied to plastic materials, and approximate both
equilibrium and traction boundary conditions more accurate. However, symmetry of the stress
and  stability of the mixed FEM make  the design of the mixed FEM for elasticity surprisingly hard, which
has been regarded as a long standing open problem \cite{Arnold-Awanou-Winther}. In fact, "Four decades of searching for mixed finite elements for elasticity beginning in the 1960s did not yield any stable elements with polynomial shape
functions" [D. N. Arnold, Proceedings of the ICM, Vol. I : Plenary Lectures and Ceremonies (2002)].

Since the 1960s, many  mathematicians have worked on this problem but compromised
to weakly symmetric elements \cite{Arnold-Brezzi-Douglas, Arnold-Falk-Winther,Boffi-Brezzi-Fortin}, or composite elements \cite{Johnson-Mercier}. In 2002, using the elasticity complexes, Arnold and Winther designed the first family of
symmetric mixed elements with polynomial shape functions on triangular grids in 2D \cite{Arnold-Winther-conforming} which was extended to tetrahedral grids in 3D \cite{Arnold-Awanou-Winther} (see  also \cite{Adams}  for the first order element in three dimensions) and to rectangular grids in 2D \cite{Arnold-Awanou}.
 Recently,  the first author and his collaborators  developed a new framework to design and analyze the mixed FEM of   elasticity
equations, which yields a family of optimal symmetric mixed FEMs. In addition, those elements are very easy to implement since their basis functions, based on those of the scalar-valued Lagrange   elements,  can be  explicitly written down by hand. The main ingredients of this framework are a structure of the discrete stress space on both simplicial and product grids, two basic algebraic results, and a two-step stability analysis method,  see more details in  \cite{Hu2015trianghigh, Hu1,HuZhang2014a, HuZhang2015tre, HuZhang2015trianglow}.

However, due to the constraint of symmetry, all the  aforementioned symmetric conforming   mixed elements on triangular meshes impose  $C^0$ continuity of discrete stresses at   internal vertices. Such  $C^0$ vertex continuity causes the failure of nestedness. In fact, for an admissible mesh $\hat{\mathcal{T}}$ which is refined from a coarse mesh $\mathcal{T}$, a newly added internal vertex $\bx_e$ of $\hat{\mathcal{T}}$ is a midpoint of a  coarse edge $e$ of $\mathcal{T}$ with the unit tangential and normal vectors $t_e$ and $n_e=t_e^\perp$, respectively. All the components of a symmetric matrix-valued function in the discrete stress space  $\Sigma(\hat{\mathcal{T}})$ on $\hat{\mathcal{T}}$ are continuous at the aforementioned vertex $\bx_e$. Whereas for a  symmetric matrix-valued function $\tau$ in the discrete stress space  $\Sigma({\mathcal{T}})$ on ${\mathcal{T}}$, two components
$n_e^T\tau n_e$ and $t_e^T\tau n_e$  are continuous  while the pure tangential component $t_e^T\tau t_e$ is not necessarily continuous at this vertex. This is to say $\tau$ is not necessarily in the fine space $\Sigma(\hat{\mathcal{T}})$ and consequently
$\Sigma({\mathcal{T}})$ is not a subspace of $\Sigma(\hat{\mathcal{T}})$. Thus there is no nestedness.

There are plenty of results about the convergence analysis of   adaptive mixed FEMs  for the mixed Poisson problem \cite{Becker2008,axioms,CarstensenHoppe2006,ChenHolstXu2009,HuYu2018,HuangXu} and \cite{HuangHuangXu2011} for Kichhoff plate bending problems.  These results are established for nested finite element spaces on adaptive meshes. The non-nestedness of the discrete stress space of the elements of \cite{Adams,Arnold-Awanou-Winther, Arnold-Winther-conforming,Hu2015trianghigh, Hu1,HuZhang2014a, HuZhang2015tre, HuZhang2015trianglow} leads to mathematical difficulty of  their adaptive algorithms, based on a posteriori error estimators of the two families of  triangular mixed elements,  see \cite{Carstensen2019,CarstensenGedicke2016,ChenHuHuangMan2017}. In fact, for the convergence analysis of standard adaptive
algorithms of non-nested conforming  methods including mixed methods,  the only positive result is the paper \cite{ZhaoHuShi2010}
where the non-nestedness  is caused by  a very particular mesh-refinement strategy, while the non-nestedness herein comes from  extra continuity of functions at the internal vertices  in the discrete stress space. It is unclear  for the authors  whether the technology of \cite{ZhaoHuShi2010} can be  extended to the current case.

One purpose of this paper is to extend a conforming triangular mixed  element recently proposed in \cite{Hu2015trianghigh,HuZhang2014a} by Hu and Zhang  for linear elasticity by rearranging the global degrees of freedom.
The underlying admissible meshes $\mathcal{T}_1$, $\cdots$, $\mathcal{T}_N$  are successively  refined from an initial mesh $\mathcal{T}_0$ through a newest vertex bisection strategy \cite{Stevenson2008}. It has already been explained that these meshes admit a crucial hierarchical structure, namely, a newly added vertex $\bx_e$ of the mesh $\mathcal{T}_\ell$  is the midpoint of an edge $e$ of the coarse mesh $\mathcal{T}_{\ell-1}$. Such a hierarchical structure will be explored to partially relax $C^0$ vertex continuity of symmetric matrix-valued functions in the discrete stress space $\Sigma(\mathcal{T}_\ell)$ on $\mathcal{T}_\ell$ of the original element and results in an extended discrete stress space $\widetilde{\Sigma(\mathcal{T}_\ell)}$.
For such an internal vertex $\bx_e$ located at the coarse edge $e$ with the unit tangential vector $t_e$ and the unit normal vector $n_e=t_e^\perp$, let $\varphi_{\bx_e}(\bx)$ be the nodal basis function of the scalar-valued Lagrange element of order $k$ ($k$ is equal to the polynomial degree of the discrete stress)  on  $\mathcal{T}_\ell$ with  $\varphi_{\bx_e}^+(\bx)$ and $\varphi_{\bx_e}^-(\bx)$ denoted its two restrictions on two sides of $e$, respectively.  The  idea to extend $\Sigma(\mathcal{T}_\ell)$ is to keep  the continuity of the normal components $\tau n_e$ of a symmetric matrix-valued function
$\tau$ in $\Sigma(\mathcal{T}_\ell)$ and to split the pure tangential component $t_e^T\tau t_e$ into two parts at vertex $\bx_e$ along edge $e$. This can be accomplished by splitting the basis function $\varphi_{\bx_e}(\bx)t_et_e^T$ into two basis functions
$\varphi_{\bx_e}^+(\bx)t_et_e^T$ and $\varphi_{\bx_e}^-(\bx)t_et_e^T$ while keeping the other two basis functions
$\varphi_{\bx_e}(\bx)n_en_e^T$  and $\varphi_{\bx_e}(\bx)(n_et_e^T+t_en_e^T)$, associated to vertex $\bx_e$.
This is,  though the extended discrete stress space $\widetilde{\Sigma(\mathcal{T}_\ell)}$ on $\mathcal{T}_\ell$ is still a $H(\div)$  subspace,  the pure tangential component along the coarse edge $e$ of discrete stresses in it is not necessarily continuous at such vertices like $\bx_e$.
Therefore, the number of the global basis functions  associated to the node $\bx_e$ of the extended discrete stress space becomes four rather than three (for the original discrete stress space). A feature of this extended discrete stress space  is its nestedness in the sense that  the space $\widetilde{\Sigma(\mathcal{T}_{\ell-1})}$ is a subspace of $\widetilde{\Sigma(\mathcal{T}_\ell)}$. Then, such a nested property is used to analyze and prove the optimal convergence of a standard adaptive algorithm.

Another purpose of the paper is to impose a general traction boundary condition on the discrete level. The underlying situation is that
 a corner $\bx_c$ of the polygonal domain $\Omega$ is the unique intersected point of two boundary edges $e_1$ and $e_2$.
 Let $t_i$ and $ n_i$ denote the unit tangential and outward normal vectors of $e_i$, $i=1, 2$. A general traction boundary condition $\sigma n_i|_{e_i}$, $i=1, 2$,  imposed on both $e_1$ and $e_2$ can be inconsistent (discontinuous) in the sense that $ n_2^T\sigma n_1|_{e_1}(\bx_c)\not=  n_1^T\sigma n_2|_{e_2}(\bx_c)$. Such inconsistency (discontinuity) causes an essential difficulty for imposing such a boundary condition if the vertex degrees of freedom of the original discrete stress space are used at the corner vertex $\bx_c$. Indeed, it holds that $ n_2^T \tau  n_1|_{e_1}(\bx_c)= n_1^T\tau n_2|_{e_2}(\bx_c)$ for any $\tau$
 in $\Sigma(\mathcal{T})$ on the mesh $\mathcal{T}$ for all the elements in \cite{Arnold-Winther-conforming,Hu2015trianghigh,HuZhang2014a}. As a result, the traction boundary condition can not be exactly
  imposed even if $\sigma n_i|_{e_i}$ is a polynomial of degree not bigger than $k$ if $ n_2^T\sigma n_1|_{e_1}(\bx_c)\not=  n_1^T\sigma n_2|_{e_2}(\bx_c)$.  To handle such a case, the authors of \cite{CarstensenGunther2008} compromised to a least  square method to obtain some approximation of the traction boundary condition.  The idea to overcome such a difficulty is to split  the triangle at the corner into two sub-triangles and then relax the continuity of the pure  tangential component across the common edge of these  two sub-triangles. This eventually  leads to  four degrees of freedom at the corner vertex. As a result, the possible inconsistent traction boundary condition is able to be imposed, in particular, be exactly imposed if
  $\sigma n_i|_{e_i}$ is a polynomial of degree not bigger than $k$.
 The numerical  results on  uniform meshes and adaptive meshes for a benchmark problem over an L-shaped domain and Cook's membrane problem indicate  that such a remedy is able to largely improve accuracy of discrete stresses especially on coarse meshes.  The reason lies in that  the error caused by the inexact boundary conditions may  dominate on  coarse meshes.

Throughout this paper,
$L^2(\omega;X)$ denotes the space of functions
which are square-integrable  with domain $\omega$. For our purposes, the range
space $X$ will be either $\mathbb{S}:=\hbox{symmetric } \mathbb{R}^{d\times d},$ $\mathbb{R}^d,$ or
$\mathbb{R}$, $d=2, 3$. Let $H^m(\omega;X)$ denote the Sobolev space consisting of
functions, taking values in the
finite-dimensional vector space $X$, and with all derivatives of
order at most $m$ square-integrable, and $H({\rm div},\omega;\mathbb{S})$
consist of square-integrable symmetric matrix fields with
square-integrable divergence. Let $\|\cdot\|_{m,\omega}$ represent the norm and $|\cdot|_{m,\omega}$ represent the seminorm of $H^m(\omega)$, and  $(\cdot,\cdot)_{\omega}$ represent as usual, the $L^2$ inner product on the domain $\omega$, the subscript $\omega$ is omitted when $\omega=\Omega$. $\langle\cdot,\cdot\rangle_\Gamma$ represents the $L^2$ inner product on the boundary $\Gamma$.
For $\phi\in H^1(\Omega;\mathbb{R})$, $v=(v_1,v_2)^T\in H^1(\Omega;\mathbb{R}^2)$, set
\begin{equation*}
  {\rm\boldsymbol{Curl}}\phi:=(-\partial\phi/\partial x_2,\partial\phi/\partial x_1),\;
  {\rm\boldsymbol{Curl}}v:=\begin{pmatrix}
                             -\partial v_1/\partial x_2 & \partial v_1/\partial x_1  \\
                             -\partial v_2/\partial x_2  & \partial v_2/\partial x_1  \\
                           \end{pmatrix}.
\end{equation*}
 For $v=(v_1,v_2)^T\in H^1(\Omega;\mathbb{R}^2)$ and $\tau=(\tau_{ij})_{2\times2}$, set
\begin{equation*}
  {\rm curl}v:=\partial v_2/\partial x_1-\partial v_1/\partial x_2,\;{\rm curl}\tau=\begin{pmatrix}
             \partial \tau_{12}/\partial x_1-\partial \tau_{11}/\partial x_2 \\
                          \partial \tau_{22}/\partial x_1-\partial \tau_{21}/\partial x_2 \\
                                                                                    \end{pmatrix}.
\end{equation*}
The notation $A\lesssim B$ abbreviates $A\leq CB$ for a mesh-size independent constant $C>0$. The context depending symbol $|\bullet|$ denotes the area of a domain, the length of a edge, the counting measure (cardinality) of a set and the absolute value of a real number.

The rest of the paper is organized as follows.   Section \ref{sec:Prelim} introduces notation in this paper and the mixed element on triangular meshes \cite{Hu2015trianghigh,HuZhang2014a}, including its degrees of freedom and basis functions.   Section \ref{sec:Adap} designs nested mixed finite elements on adaptive meshes by partially relaxing the $C^0$ vertex continuity of discrete stresses at newly added internal vertices of refined meshes and  proves  optimality of the corresponding adaptive algorithm.  In Section \ref{sec:allcorner},   the $C^0$ vertex continuity is relaxed at corner vertices of domain $\Omega$  so that an inconsistent traction boundary condition on $\Gamma_N$ can be imposed. This section will give some comments to the three dimensional case. Numerical experiments are presented in Section~\ref{Numerics}.

\section{Preliminaries}\label{sec:Prelim}This section  introduces the stress-displacement  formulation  for   linear elasticity and the mixed FEM  from \cite{Hu2015trianghigh,HuZhang2014a}.
\subsection{Mixed formulation }
Let $\Omega\subset\mathbb{R}^2$ be a simply-connected bounded polygonal domain with boundary $\Gamma:=\partial\Omega=\Gamma_D\cup\Gamma_N$, $\Gamma_D\cap\Gamma_N=\emptyset$.
Given $f\in V:=L^2(\Omega;\mathbb{R}^2)$, $u_D\in H^1(\Omega;\mathbb{R}^2)$ and $g\in L^2(\Gamma_N;\mathbb{R}^2)$, the linear elasticity problem
with mixed boundary conditions within a stress-displacement formulation
reads: Seek $(\sigma,u)\in\Sigma_g\times V$ such that
\an{\left\{ \ad{
  (A\sigma,\tau)+({\rm div}\tau,u)&= \langle u_D,\tau \normalvec\rangle_{\Gamma_D} && \hbox{for all \ } \tau\in\Sigma_0,\\
   ({\rm div}\sigma,v)&= (f,v) &\qquad& \hbox{for all \ } v\in V}
   \right.\lab{eqn1}
}
with
\begin{equation*}
W:=\{v\in H^1(\Omega;\mathbb{R}^2)\ |\ v|_{\Gamma_D}=0\},
\end{equation*}
\begin{equation*}
  \Sigma_g:=\{\sigma\in H(\ddiv,\Omega;\mathbb{S})\ |\ \langle\psi,\sigma \normalvec\rangle_{\Gamma_N}=\langle\psi,g\rangle_{\Gamma_N}\text{ for all }\psi\in W\}.
\end{equation*}
Here $\Sigma_0:=\Sigma_g$ with $g\equiv0$ and  $n$ denotes the outnormal of $\partial\Omega$.
The compliance tensor $A:\mathbb{S}\rightarrow\mathbb{S}$, characterizing the properties of the material, is symmetric positive definite and its eigenvalues are uniformly bounded from above. In the homogeneous isotropic case, the compliance tensor is given by $A\tau=\big(\tau-\lambda/(2\mu+2\lambda){\rm tr}\tau {\rm I}\big)/(2\mu)$ with the Lam\'{e} constants $\mu>0$ and $\lambda\geq 0$, the identity matrix ${\rm I}$ and the trace ${\rm tr}\tau$ of the matrix $\tau$. For simplicity, this paper assumes $A$ is a constant tensor.
\subsection{Triangulations}\label{sec:triangulation}
 Given an initial shape-regular triangulation $\cT_0$ of $\Omega$ into triangles, let $\mathbb{T}:=\mathbb{T}(\cT_0)$ denote the set of all admissible regular triangulations obtained from $\T_0$ with a finite number of successive bisections of triangles by the newest vertex bisection \cite{Stevenson2008}. Given $\T\in\mathbb{T}$, let $\hat{\T}$ denote a refinement of $\T$, and $\T\setminus\hat{\T}:=\{K\in\T\ |\ K\not\in\hat{\T}\}$ denote  the set of refined elements from $\T$ to $\hat{\T}$.
  Let $h_K:=|K|^{1/2}$ and $h=\max_{K\in\hat{\T}}h_K$. Denote by  $\hat{\mathcal{E}}$  (resp. $\hat{\mathcal{E}}(\Omega)$ and  $\hat{\mathcal{E}}(\Gamma)$) the collection of all (resp. interior and boundary) element edges of $\hat{\T}$. For any triangle $K\in\hat{\T}$, let $\mathcal{E}(K)$ denote the set of its edges. For any edge $e\in\hat{\mathcal{E}}$,  let  $t_e$ denote the unit tangential vector and let  $n_e:=t_e^\perp$ denote the unit normal vector. If $e\in\hat{\cE}(\Gamma)$, $n_e=n$ is the outward unit normal. The jump $[w]_e$ of $w$ across edge $e=K_1\cap K_2$ reads
\begin{equation*}
  [w]_e:=(w|_{K_1})|_e-(w|_{K_2})|_e.
\end{equation*}
Particularly, if $e\in\hat{\mathcal{E}}(\Gamma)$, $ [w]_e:=w|_e$.  Let $\mathcal{V}(\hat{\T})$ (resp. $\mathcal{V}_0(\hat{\T})$) denote the set of  all (resp. internal) vertices of  $\hat{\T}$.    The newest-vertex bisection (NVB) (see  more details in \cite{Stevenson2008}) creates
 each  new   vertex $\bx_e \in\mathcal{V}(\hat{\T})\setminus\mathcal{V}(\T_0)$ as a  midpoint of some   edge $e $ associated with tangential  vector  $ \tangevec_e$ and normal vector  $ \normalvec_e$.  Given any internal node $\bx_e  \in\mathcal{V}_0(\hat{\T})\setminus\mathcal{V}(\T_0)$, define the two patches $\omega^+_{\bx_e}$ and $\omega^-_{\bx_e}$ by \begin{align}\label{KPos}
\begin{aligned}
\omega^+_{\bx_e}:&=\cup\{K\ |\ K\in\hat{\T}, \bx_e\in K,  ({\rm mid}(K)-\bx_e)\cdot n_{e}>0\},\\
\omega^-_{\bx_e}:&=\cup\{K\ |\ K\in\hat{\T},  \bx_e\in K,({\rm mid}(K)-\bx_e)\cdot n_{e}<0\}.
\end{aligned}
\end{align}

Given any integer $k\geq 0$, let $P_k(\omega;X)$ denote the space of polynomials over $\omega$ of total degree not greater than $k$, taking values in the finite-dimensional vector space X.  Let $\bx_i,\  1\leq i\leq 3 $ denote the vertices of $K\in\hat{\T}$, $\lambda_i$ denote the barycentric coordinates with respect to $\bx_i$, and $t_{i,j}=\bx_j-\bx_i$ denote the tangential vector of edge $\bx_i\bx_j$.

\subsection{Mixed finite element method}\label{sec:orignalMFEM}
Given $K\in\hat{\T}$,
 with symmetric matrices $\mathbb{S}_{i,j}:=t_{i,j}t_{i,j}^T$, $1\leq i< j\leq 3$ of rank one, define the following  space
 \cite{Hu2015trianghigh,HuZhang2014a}:
 \begin{equation*}
   \Sigma_{k,b}(K):=\sum_{1\leq i<j\leq 3}\lambda_i\lambda_jP_{k-2}(K;\mathbb{R})\mathbb{S}_{i,j}.
 \end{equation*}
 Note that  for any function $\tau\in\Sigma_{k,b}(K)$ and  any edge $e$ of $K$  with the normal vector $n_e$, the normal components $\tau n_e|_{e}$ vanish. This is to say $\Sigma_{k,b}(K)$ is  an $H({\rm div},K;\mathbb{S})$ bubble function space on $K$.
 With such a bubble function space on each element of $\hat{\T}$, one can define a discrete stress space $\Sigma(\hat{\T})$ with a simple and crucial structure on $\hat{\T}$, i.e, for $k\geq 3$,
 \begin{equation}\label{Def:Odisstress}
 \begin{split}
   \Sigma(\hat{\T}):=&\big\{\sigma\in H({\rm div},\Omega;\mathbb{S})\ |\ \sigma=\sigma_c+\sigma_b,\sigma_c\in H^1(\Omega;\mathbb{S}),\\
   &\forall K\in\hat{\T}, \sigma_c|_K\in P_k(K;\mathbb{S}),\sigma_b|_K\in\Sigma_{k,b}(K) \big\}.
   \end{split}
 \end{equation}
 Note that $\sigma_c$ is a symmetric matrix-valued $H^1$ Lagrange element function, this is, each component of $\sigma_c$ is
 a scalar-valued   Lagrange element function on $\hat{\T}$. This is, the discrete stress space $\Sigma(\hat{\T})$ is a sum of the symmetric matrix-valued $H^1$ Lagrange element function space and these $H(\rm div)$ bubble function spaces.  While the bubble function space $\Sigma_{k,b}(K)$ on element $K$ plays
a very important role in the two step stability analysis of the element, see  \cite{Hu2015trianghigh,HuZhang2014a} for more details.

Next, the degrees of freedom will be presented for the stress shape function space $P_k(K;\mathbb{S})$ on element $K$. Indeed,
 a symmetric matrix field $\tau\in P_k(K;\mathbb{S})$ can be uniquely determined by the degrees of freedom from (1), (2) and (3) (see Figure~\ref{fig:lowcaseN}  with solid points and arrows for $k=3$) \cite{Hu2015trianghigh}:
 \begin{enumerate}
   \item the values of  $\tau$ at three vertices,
   \item for edge $e$, the mean moments of degree at most $k-2$ over $e$ of $n_e\tau n_e^T$, $t_e^T\tau n_e$,
   \item the values  $\int_K\tau:\xi\,dx$ for any $\xi\in\Sigma_{k,b}(K)$.
 \end{enumerate}
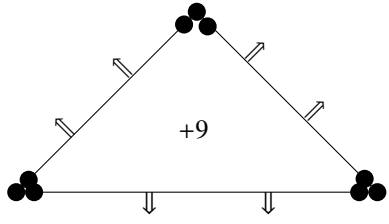
\begin{figure}[h!]
 {\setlength{\unitlength}{2.4cm}
\begin{picture}(2,1)
\put(0,0){\line(1,0){2}}
\put(0,0){\line(1,1){1}}
\put(2,0){\line(-1,1){1}}
\put(0,0){\circle*{0.1}}  \put(0.07,0.07){\circle*{0.1}}  \put(0.1,0){\circle*{0.1}}
\put(2,0){\circle*{0.1}}  \put(1.9,0){\circle*{0.1}}  \put(1.93,0.075){\circle*{0.1}}
 \put(1,1){\circle*{0.1}} \put(0.93,0.93){\circle*{0.1}} \put(1.06,0.92){\circle*{0.1}}
  \put(0.90,0.3){$ +9$}
\put(0.21,0.33){\footnotesize $\Nwarrow$}
\put(0.53,0.66){\footnotesize $\Nwarrow$}
 \put(0.66,-0.1){ $\Downarrow$}
 \put(1.32,-0.1){ $\Downarrow$}
\put(1.26,0.75){\footnotesize $\Nearrow$}
\put(1.59,0.42){\footnotesize $\Nearrow$}
\end{picture}}
 {\setlength{\unitlength}{1cm}}
\caption{Degrees of freedom for $\Sigma(\hat{\T})$ of $k=3$}\label{fig:lowcaseN}
\end{figure}

For ease implementation, the global basis functions of $\Sigma(\hat{\T})$ when $k=3$ will be presented, which needs the nodal basis functions of the scalar-valued Lagrange element of order $3$, which, on element $K$,  read as follows
\begin{equation*}
\begin{split}
\varphi_0(\bx)&=27\lambda_1\lambda_2\lambda_3,\\
\varphi_i(\bx)&=\frac{9}{2}\lambda_i(\lambda_i-\frac{1}{3})(\lambda_i-\frac{2}{3}), i=1, 2,3,\\
\varphi_{ij}(\bx)&=\frac{27}{2}\lambda_{i+1}\lambda_{i+2}(\lambda_{i+j}-\frac{1}{3}), i=1,2,3,j=1,2.
\end{split}
\end{equation*}
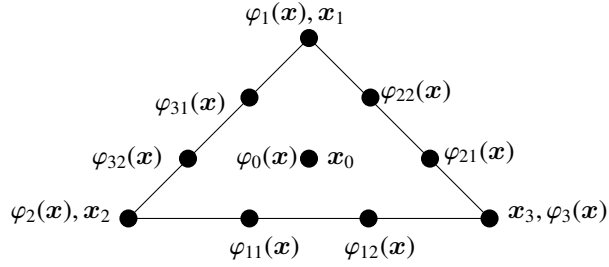
\begin{figure}[h!]
 {\setlength{\unitlength}{2.40cm}
\begin{picture}(2,1.3)
\put(0,0){\line(1,0){2}}
\put(-0.65,0.0){$\varphi_2(\bx), \bx_2$}
\put(0,0){\line(1,1){1}}
\put(2,0){\line(-1,1){1}}
\put(2.1,0.0){$\bx_3, \varphi_3(\bx)$}
\put(0,0){\circle*{0.1}} 
\put(2,0){\circle*{0.1}}  
 \put(1,1){\circle*{0.1}} 
\put(0.65,1.1){$\varphi_1(\bx), \bx_1$}
  \put(1.0,0.33){\circle*{0.1}}
\put(-0.20,0.3){$\varphi_{32}(\bx)$}
\put(0.33,0.33){\circle*{0.1}}
\put(0.67,0.67){\circle*{0.1}}
\put(0.15,0.60){$\varphi_{31}(\bx)$}
\put(1.33,0){\circle*{0.1}}
\put(0.67,0){\circle*{0.1}}
\put(1.2,-0.2){$\varphi_{12}(\bx)$}
\put(0.56,-0.2){$\varphi_{11}(\bx)$}
\put(1.34,0.67){\circle*{0.1}}
\put(1.67,0.33){\circle*{0.1}}
\put(1.4,0.67){$\varphi_{22}(\bx)$}
\put(1.75,0.33){$\varphi_{21}(\bx)$}
 \put(0.60,0.3){$\varphi_{0}(\bx)$}
 \put(1.1,0.3){$\bx_0$}
\end{picture}
}
\\
\ \\
\caption{Ten nodal basis functions of the Lagrange element of order $3$ on element $K$.}\label{Lagrange}
\end{figure}

Two sets of basis of the symmetric matrix space $\mathbb{S}$ are required as well. The first set of basis is the canonical basis of $\mathbb{S}$:
\begin{equation}
  \mathbb{S}_1=\p{1 & 0\\ 0&0}, \mathbb{S}_2=\p{0& 1\\1 &0}, \text{ and }\mathbb{S}_3=\p{0&0\\0 &1}.
  \end{equation}
The second set of basis of $\mathbb{S}$ is defined associated to each edge. More precisely, given edge $e_i$ of element $K$ , define  $\mathbb{S}_{e_i}:=t_{e_i}t_{e_i}^T$ and $\mathbb{S}_{e_i,m}^\perp$ with $m=1, 2$ such that
\begin{equation}
\mathbb{S}_{e_i, 1}^\perp:=n_{e_i}n_{e_i}^T, \mathbb{S}_{e_i, 2}^\perp:=n_{e_i}t_{e_i}^T+t_{e_i}n_{e_i}^T,\text{ therefore, }\mathbb{S}_{e_i, m}^\perp: \mathbb{S}_{e_i}=0
 \text{ and }\mathbb{S}_{e_i, 1}^\perp: \mathbb{S}_{e_i, 2}^\perp=0.
\end{equation}
Then, the global basis functions of the discrete stress space associated to element $K$ are as follows:
\begin{equation}\label{eq:bas}
\begin{split}
\theta_{ij}&=\varphi_i(\bx)\mathbb{S}_j, i=0, 1,2,3, j=1,2,3,\\
\alpha_{ij}&=\varphi_{ij}(\bx)|_K\mathbb{S}_{e_i}, i=1,2,3,j=1,2,\\
\beta_{ijm}&=\varphi_{ij}(\bx)\mathbb{S}_{e_i,m}^\perp, i=1,2,3,j=1,2, m=1,2.
\end{split}
\end{equation}
Here $\varphi_{ij}(\bx)|_K$ denotes the restriction of $\varphi_{ij}(\bx)$ on element $K$. 
Note that the basis function $\alpha_{ij}$ is an $H(\rm div)$ bubble function in the sense that $\alpha_{ij} n_{e}|_{e}=0$ on any edge $e$ of element $K$. Furthermore, if $e_i$ is a common edge of element $K$ and element $K^\prime$, the corresponding $H(\rm div)$ bubble function with respect to element $K^\prime$ is $\alpha_{ij}^\prime=\varphi_{ij}(\bx)|_{K^\prime}\mathbb{S}_{e_i}$ which is also a basis function of $\Sigma(\hat{\T})$ and independent of $\alpha_{ij}$. It is stressed that a basis function of $\Sigma(\hat{\T})$ is the product of a scalar-valued Lagrange element basis function (or the corresponding restriction on an element ) and a basis of the symmetric matrix space $\mathbb{S}$.

 The displacement space is the full $C^{-1}$-$P_{k-1}$ space
 \begin{equation}\label{displacementspace}
   V(\hat{\T}): =\{v\in L^2(\Omega;\mathbb{R}^2)\ |\ v|_K\in P_{k-1}(K;\mathbb{R}^2)\text{ for all }K\in\hat{\T}\}.
 \end{equation}
Suppose $g({\hat{\T}}): =\alpha({\hat{\T}}) n|_{\Gamma_N}$ for some $\alpha({\hat{\T}})\in\Sigma(\hat{\T})$ denotes some approximation to $g$.  The mixed FEM to \eqref{eqn1}  seeks $(\sigma(\hat{\T}),u(\hat{\T}))\in \Sigma(\hat{\T})\cap\Sigma_{g(\hat{\T})}\times V(\hat{\T})$ such that
 \an{\left\{ \ad{
  (A\sigma(\hat{\T}),\tau(\hat{\T}))+({\rm div}\tau(\hat{\T}),u(\hat{\T}))&= \langle u_D,\tau(\hat{\T}) n\rangle_{\Gamma_D} && \hspace{-5mm}\hbox{for all \ } \tau(\hat{\T})\in\Sigma(\hat{\T})\cap\Sigma_{0},\\
   ({\rm div}\sigma(\hat{\T}),v(\hat{\T}))&= (f,v(\hat{\T})) &\qquad&\hspace{-5mm} \hbox{for all \ } v(\hat{\T})\in V(\hat{\T}). }
   \right.\lab{disMixElas}
}
\section{Adaptive mixed finite element methods}\label{sec:Adap} Due to the vertex $C^0$ continuity   of functions  in  $\Sigma(\hat{\T})$  from  \eqref{Def:Odisstress}, as it was explained in the introduction that  the discrete stress space  is non-nested in the sense that a coarse space $\Sigma(\T)$ is not a subspace of a fine space $\Sigma(\hat{\T})$ when $\hat{\T}$ is an admissible refinement of $\T$ by NVB.
 In fact, a newly added internal vertex $\bx_e$ of $\hat{\mathcal{T}}$ is a midpoint of an coarse edge $e$ of $\mathcal{T}$ with the unit tangential and normal vectors $t_e$ and $n_e=t_e^\perp$, respectively. 
 The vertex $C^0$ continuity  implies that  all the components of  functions in   $\Sigma(\hat{\mathcal{T}})$ are continuous at  vertex $\bx_e$. However, for a  function $\tau$ in  $\Sigma({\mathcal{T}})$, its pure tangential component $t_e^T\tau t_e$ is not necessarily continuous at this vertex. This is to say $\tau$ is not necessarily in the fine space $\Sigma(\hat{\mathcal{T}})$ and consequently $\Sigma({\mathcal{T}})$ is not a subspace of $\Sigma(\hat{\mathcal{T}})$.  With such  non-nestedness, it is  actually very difficult to analyze the optimal convergence of the corresponding  adaptive algorithm. In the adaptive algorithm, the underlying admissible meshes $\mathcal{T}_1$, $\cdots$, $\mathcal{T}_N$, based on some a posteriori error estimate, see for instance \cite{ChenHuHuangMan2017},   are successively  refined from an initial mesh $\mathcal{T}_0$ through a NVB strategy \cite{Stevenson2008}. Since $\Sigma(\T_\ell)$ is usually not a subspace of $\Sigma(\T_m)$ with $m>\ell$, it is essentially difficult to show some orthogonality or quasi-orthogonality which is a main ingredient for the optimal convergence
 analysis of adaptive algorithms \cite{axioms,HuYu2018}. For the non-nestedness caused by the extra smoothness of the underlying finite element space, there is no convergence analysis of the corresponding adaptive algorithm in the literature so far.

 This section relaxes the $C^0$ vertex continuity of functions in $\Sigma(\hat{\T})$ to design  an extended stress space $\widetilde{\Sigma(\hat{\T})}$ when $\hat{\T}$ is a refinement of $\T$ by NVB.
 For the admissible nested meshes $\mathcal{T}_0$, $\mathcal{T}_1$, $\cdots$, $\mathcal{T}_N$ from the adaptive algorithm, 
 this leads to a sequence of nested spaces $\Sigma(\T_0)$,  $\widetilde{\Sigma({\T_1})}$, $\cdots$, $\widetilde{\Sigma({\T_N})}$.
 The optimal convergence of the adaptive algorithm based on these nested spaces will be proved, adopted to the unified analysis from \cite
{HuYu2018}.  For simplicity, this section only considers the homogeneous Dirichlet boundary condition $\Gamma_D=\Gamma$ and $u_D\equiv0$.

\subsection{Extended stress space  on adaptive meshes}\label{sec:extend}The extended stress space $\widetilde{\Sigma(\hat{\T})}$ extends the global degrees of freedom (1) of  $\Sigma(\hat{\T})$ from Subset.~\ref{sec:orignalMFEM}  to make it hierarchical and defined for an admissible triangulation $\hat{\T}\in\TT$ as follows. Recall that any internal vertex $\bx_e\in\mathcal{V}_0(\hat{\T})\setminus\mathcal{V}(\T_0)$ ($e$ is a coarse edge and $\bx_e$ is its midpoint) is associated to the separation $\omega_{\bx_e}^+$ and $\omega_{\bx_e}^-$  of the neighbouring triangles with vertex $\bx_e$. Instead of the continuity of all components of $\tau$ at $\bx_e$, the component  $\tangevec_{ e}^T\tau\tangevec_{ e}$  at $\bx_e$ is not uniquely defined at $\bx_e$ but allows for one value  in $\omega_{\bx_e}^+$ and a second value in $\omega_{\bx_e}^-$. Meanwhile,  the vertex-based basis functions associated to $\bx_e$ are enriched  to  four basis functions:
\begin{equation}\label{eq:extbas}
  \varphi_{\bx_e}(\bx)n_{ e}n_{ e}^T,\; \varphi_{\bx_e}(\bx)(n_{ e}t_{ e}^T+t_{ e}n_{ e}^T)
,\; \tau_{\bx_e}^+:=\varphi_{\bx_e}^+(\bx)t_{ e}t_{ e}^T\text{ and } \tau_{\bx_e}^-:=\varphi_{\bx_e}^-(\bx)t_{ e}t_{ e}^T .
    \end{equation}
Here $\varphi_{\bx_e}^+(\bx)=\varphi_{\bx_e}(\bx)$ for $\bx\in\omega_{\bx}^+$ and otherwise vanishes, and $\varphi_{\bx_e}^-(\bx)$ is similarly defined for $\omega_{\bx_e}^-$.

Let $E(\hat{\T}):={\rm span}_{\bx_e\in\mathcal{V}_0(\hat{\T})\setminus\mathcal{V}(\mathcal{T}_0)}\{ \tau_{\bx_e}^+, \tau_{\bx_e}^-\}$.
The extended  stress space  is then defined by
\begin{align}\label{eq:exspace}
\widetilde{\Sigma(\hat{\T})}:=\Sigma(\hat{\T})+E(\hat{\T}).
\end{align}
Note that for any $\tau(\hat{\T})\in E(\hat{\T})$ and any $e\in\hat{\mathcal{E}}(\Omega)$, the normal components   $\tau(\hat{\T})\normalvec_e$ are  continuous across $e$. This guarantees  $ E(\hat{\T})\subset H(\ddiv,\Omega;\mathbb{S})$ and thus $\widetilde{\Sigma(\hat{\T})}\subset  H(\ddiv,\Omega;\mathbb{S})$.
The displacement is  approximated by  $V(\hat{\T})$ in \eqref{displacementspace}. Then the extended mixed FEM to \eqref{eqn1}  with homogeneous Dirichlet boundary condition seeks $(\sigma(\hat{\T}),u(\hat{\T}))\in \widetilde{\Sigma(\hat{\T})} \times V(\hat{\T})$ such that
 \an{\lab{disMixElasextend}\left\{ \ad{
  (A\sigma(\hat{\T}),\tau(\hat{\T}))+({\rm div}\tau(\hat{\T}),u(\hat{\T}))&=0&& \hbox{for all \ } \tau(\hat{\T})\in\widetilde{\Sigma(\hat{\T})} ,\\
   ({\rm div}\sigma(\hat{\T}),v(\hat{\T}))&= (f,v(\hat{\T})) &\qquad& \hbox{for all \ } v(\hat{\T})\in V(\hat{\T}). }
   \right.
}
For ease of notation, throughout this section let $(\sigma(\hat{\T}),u(\hat{\T}))\in \widetilde{\Sigma(\hat{\T})} \times V(\hat{\T})$ denote the solution to \eqref{disMixElasextend} instead of \eqref{disMixElas}.
\begin{theorem}For $k\geq3$,
the extended discrete problem \eqref{disMixElasextend} has a unique solution $(\sigma(\hat{\T}),u(\hat{\T}))\in \widetilde{\Sigma(\hat{\T})} \times V(\hat{\T})$ and the error estimate holds
\begin{align}\label{eq:errest}
\|\sigma-\sigma(\hat{\T})\|_{H(\ddiv)}+\|u-u(\hat{\T})\|_0\leq Ch^k(\|\sigma\|_{k+1}+\|u\|_k).
\end{align}
\end{theorem}
\begin{proof}The  well-posedness and the  error estimate  of    \eqref{disMixElas} can be found in \cite{Hu2015trianghigh,HuZhang2014a}.  Since $\ddiv\Sigma(\hat{\T})=V(\hat{\T})$ and   $\ddiv E(\hat{\T})\subset V(\hat{\T})$ imply $\ddiv\widetilde{\Sigma(\hat{\T})}=V(\hat{\T})$,  the combination with $\Sigma(\hat{\T})\subset \widetilde{\Sigma(\hat{\T})}$ leads to the well-posedness of \eqref{disMixElasextend} and its error estimate \eqref{eq:errest}  (see e.g. \cite[Prop.~5.4.1]{Boffi-Brezzi-Fortin}).
\end{proof}
\begin{theorem}[nestedness]Given any $\T\in\TT$
and its refinement $\hat{\T}$,  the  extended stress space with respect to $\T$ satisfies $\widetilde{\Sigma(\T)}\subset\widetilde{\Sigma(\hat{\T})}$.
\end{theorem}
\begin{proof}Given an  interior edge  $e\in\mathcal{E}(\Omega)$  with the attached two triangles $K_j\in\T$ for $j=1,2$,
the  bisection of  $e $ and $K_j$ leads to four new global degrees of freedom at the midpoint $\bx_e:={\rm mid}(e)\in \mathcal{V}_0(\hat{\T})\setminus\mathcal{V}(\T)$. Given $\tau(\T)\in \widetilde{\Sigma(\T)}$, the polynomial $\tau(\T)|_{K_j}\in P_k(K_j;\bS)$  is continuous along $e$ and the normal component is globally continuous at $\bx_e$. Consequently, $\tau(\T)$ can be represented with the basis functions in \eqref{eq:extbas} and the remaining basis functions of $\Sigma(\hat{\T})$     (e.g.~\eqref{eq:bas} for $k=3$). This concludes the proof.
\end{proof}
\subsection{Error estimator and adaptive algorithm}\label{sec:errada}
To establish the adaptive algorithm,  this paper employs  the a posteriori error estimator $\eta^2(\hat{\T}):=\sum_{K\in\hat{\T}}\eta^2(\hat{\T},K)$ from  \cite{ChenHuHuangMan2017} with
\begin{align}\label{eq:estimator}
\eta^2(\hat{\T},K): = h_K^4\|{\rm curl curl \,}(A\sigma(\hat{\T}))\|_{0,K}^2+\sum_{e\in\cE(K)}\Big(h_{K}\|\mathcal{J}_{e,1}\|_{0,e}^2+h_K^3\|\mathcal{J}_{e,2}\|_{0,e}^2\Big)
\end{align}
and
$$
\begin{array}{ll}
 \disp \mathcal{J}_{e,1}:=&\left\{\begin{array}{ll}
 \disp\Big[(A\sigma(\hat{\T}))\tangevec_e\cdot\tangevec_e\Big]_e&{\rm if~} e\in\hat{\mathcal{E}}(\Omega),\\
 \disp\Big((A\sigma(\hat{\T}))\tangevec_e\cdot\tangevec_e \Big)\big|_e\quad\quad\quad\quad\quad\ & {\rm if~} e\in\hat{\mathcal{E}}(\Gamma),
 \end{array}\right. \\
 \\
 \disp \mathcal{J}_{e,2}:=&\left\{\begin{array}{lr}
\disp \Big[{\rm curl}(A\sigma(\hat{\T}))\cdot \tangevec_e\Big]_e\quad&\quad\hspace{1mm} {\rm if~} e\in\hat{\mathcal{E}}(\Omega),\\
\disp \Big({\rm curl}(A\sigma(\hat{\T}))\cdot \tangevec_e -\partial _{\tangevec_e}\big((A\sigma(\hat{\T}))\tangevec_e\cdot \normalvec_e\big)\Big)\big|_e& \hspace{-1mm} {\rm if~} e\in\hat{\mathcal{E}}(\Gamma).\hspace{0.5mm}
 \end{array}\right.
 \end{array}
$$

Let $\eta^2(\hat{\T},\mathcal{M}):=\sum_{K\in\mathcal{M}}\eta^2(\hat{\T},K)$ for all $\mathcal{M}\subseteq\hat{\T}$. Let $Q_{\hat{\T}}$ denote  the  $L^2$ orthogonal projection   onto  $V(\hat{\T})$. The data oscillation reads
\begin{equation*}
 {\rm osc}^2(f,\mathcal{M}):=\sum_{K\in\mathcal{M}}h^2_K\|f-Q_{\hat{\T}}f\|^2_{0,K}.
\end{equation*}
 Recall the solution  $(\sigma,u)$   to \eqref{eqn1} with $\Gamma_D=\Gamma$ and $u_D\equiv0$ and the solution $(\sigma(\hat{\T}),u(\hat{\T}))\in\widetilde{\Sigma(\hat{\T})}\times V(\hat{\T})$   to \eqref{disMixElasextend} with respect to  $\hat{\T}$.  Define the weighted norm $\|\bullet\|_A:=(A\bullet,\bullet)^{1/2}$ in $L^2(\Omega;\bS)$.
 \begin{theorem}[reliability and efficiency]There exist positive constants $C_{Rel}$ and $C_{Eff}$ depending on the shape regularity of $\hat{\T}$ such that
\begin{equation}\label{reliability}
  \|\sigma-\sigma(\hat{\T})\|^2_A\leq C_{Rel}(\eta^2( \hat{\T})+{\rm osc}^2(f,\hat{\T}))\ \text{ (reliability)},
\end{equation}
\begin{equation}\label{efficiency}
 \eta^2( \hat{\T})\leq C_{Eff}\|\sigma-\sigma(\hat{\T})\|^2_A\ \text{ (efficiency)}.
\end{equation}
\end{theorem}
\begin{proof}The reliability and efficiency of the  error estimator for the mixed fnite element method \cite{Hu2015trianghigh,HuZhang2014a} in Subsect.~\ref{sec:orignalMFEM} have been investigated  in \cite{ChenHuHuangMan2017}.  Note that  the extended stress space $\widetilde{\Sigma(\hat{\T})}$ only differs from the stress space $\Sigma(\hat{\T})$ in \eqref{Def:Odisstress}  in the vertex degrees of freedom. Therefore,  an analogy argument in   \cite[Thm.~3.1]{ChenHuHuangMan2017} can be applied to prove the reliability \eqref{reliability}, and  in  \cite[Thm.~3.2]{ChenHuHuangMan2017} can  be applied to prove the efficiency \eqref{efficiency}. The details are omitted  in this paper. Nevertheless, the discrete reliability in Theorem~\ref{thm:disrelia} below will lead to the reliability.
\end{proof}

\setlength{\intextsep}{2pt}
\renewcommand{\thealgocf}{}
Let $\widetilde{\Sigma(\T_\ell)}$ denote the space in \eqref{eq:exspace}  and $V(\T_\ell)$ denote the space  in   \eqref{displacementspace}  with respect to   the triangulation $\T_\ell$ for $ \ell\in\mathbb{N}_0$ from the following adaptive algorithm.
\begin{algorithm}
\label{alg:amfem}
\caption{Adaptive algorithm for the nested mixed FEM}
Given a parameter $0<\theta<1$ and an initial triangulation
$\mathcal{T}_{0}$. Set $ \ell:=0$.
\begin{itemize}
    \item\textbf{SOLVE}: Solve \eqref{disMixElasextend} with respect to  $\mathcal{T}_\ell$ for
    the  solution $(\sigma(\T_{\ell}),u(\T_{\ell}))\in\widetilde{\Sigma(\T_\ell)}\times
V(\T_\ell)$.
    \item \textbf{ESTIMATE}: Compute the error indicator $\eta^2({\mathcal{T}_\ell})$  of \eqref{eq:estimator}.
    \item \textbf{MARK}: Mark a set $\mathcal{M}_\ell\subset\mathcal{T}_\ell$ with  (almost) minimal cardinality with
    \[
     \theta \big(\eta^2({\mathcal{T}_\ell})+ {\rm osc}^2(f,\cT_\ell)\big)\leq    \eta^2(\T_\ell,{\mathcal{M}_\ell})+ {\rm osc}^2(f,  \mathcal{M}_\ell ).
            \]
    \item \textbf{REFINE}: Refine each triangle $K$   in $\mathcal{M}_\ell $ by  NVB to get
    $\mathcal{T}_{\ell+1}$.
    \item Set $\ell:=\ell+1$ and go to Step \textbf{SOLVE}.
\end{itemize}
\end{algorithm}
\subsection{Optimal convergence} Since the residual-based  estimator is employed, the main task of optimality analysis is to prove   the quasi-orthogonality and the discrete reliability.

 Let $(\sigma(\T),u(\T))\in \widetilde{\Sigma(\T)}\times V(\T)$ (resp. $(\sigma(\hat{\T}),u(\hat{\T}))\in \widetilde{\Sigma(\hat{\T})}\times V(\hat{\T})$) solve \eqref{disMixElasextend} with respect to $\T\in\TT$ (resp. its refinement $\hat{\T}$).  Recall the $L^2$ projection $Q_\T $ (resp. $Q_{\hat{\T}}$) onto $V(\T)$ (resp. $V(\hat{\T})$). The analysis of the quasi-orthogonality and the discrete reliability requires the intermediate solution $(\hat{\sigma}(\hat{\T}),\inteu(\hat{\T}))\in\widetilde{\Sigma(\hat{\T})}\times V(\hat{\T})$ with
 \an{\lab{disMixElasextendauxi}\left\{ \ad{
  (A\hat{\sigma}(\hat{\T}),\tau(\hat{\T}))+({\rm div}\tau(\hat{\T}),\inteu(\hat{\T}))&=0&& \hbox{for all \ } \tau(\hat{\T})\in\widetilde{\Sigma(\hat{\T})} ,\\
   ({\rm div}\hat{\sigma}(\hat{\T}),v(\hat{\T}))&= (Q_\T f,v(\hat{\T})) &\qquad& \hbox{for all \ } v(\hat{\T})\in V(\hat{\T}). }
   \right.
}

\begin{lemma}\label{lem:intepost}
Let $(\sigma(\hat{\T}),u(\hat{\T}))\in (\widetilde{\Sigma(\hat{\T})},V(\hat{\T}))$ solve   \eqref{disMixElasextend}, and let $(\hat{\sigma}(\hat{\T}),\inteu(\hat{\T}))\in(\widetilde{\Sigma(\hat{\T})},V(\hat{\T}))$ solve  \eqref{disMixElasextendauxi}. Then it holds
\begin{equation*}
  \|\sigma(\hat{\T})-\hat{\sigma}(\hat{\T})\|_A\lesssim {\rm osc}(f, {\T\backslash\hat{\T}}).
\end{equation*}
\end{lemma}
\begin{proof}Given any $f\in L^2(\Omega;\R^2)$,
write the mixed formulation for linear elasticity \eqref{eqn1} with $\Gamma_D=\Gamma$ and $u_D\equiv0$ as $\mathcal{L}(\sigma,u)=f$.  Let $(\xi,z)=\mathcal{L}^{-1}\big((1-Q_\T )Q_{\hat{\T}} f\big)$. A stability result as  in  \cite[Lemma~3.1]{ChenHuHuangMan2017}  shows
\begin{align}
\label{eq:conintest}
\|\xi\|_A\lesssim {\rm osc}(Q_{\hat{\T}}f,\T)={\rm osc}(Q_{\hat{\T}}f,\T\setminus\hat{\T}).
\end{align}
Note that $(\sigma(\hat{\T})-\hat{\sigma}(\hat{\T}),u(\hat{\T})-\inteu(\hat{\T}))$ is an approximation to $(\xi,z)$ with respect to $\hat{\T}$.  The best $L^2$ approximation \cite{Boffi-Brezzi-Fortin} leads to
\begin{align*}
\|\xi-(\sigma(\hat{\T})-\hat{\sigma}(\hat{\T}))\|_A\lesssim\inf_{\tau(\hat{\T})\in\widetilde{\Sigma(\hat{\T})}}\|\xi-\tau(\hat{\T})\|_A\leq\|\xi\|_A.
\end{align*}
This, \eqref{eq:conintest}  and a triangle inequality conclude the proof.
\end{proof}
\begin{theorem}[quasi-orthogonality]\label{thm:quasiotho}
For any $0<\delta<1$, there exists $C_0>0$ such that
\begin{align*}
(1-\delta)\|\sigma -\sigma(\hat{\T})\|_A^2\leq \|\sigma -\sigma(\T)\|_A^2-\|\sigma(\hat{\T}) -\sigma(\T)\|_A^2+\frac{C_0}{\delta}{\rm osc}^2(f,\T\setminus\hat{\T}).
\end{align*}
\end{theorem}
\begin{proof}Recall $\hat{\sigma}(\hat{\T})$ from \eqref{disMixElasextendauxi}. Since $\ddiv(\hat{\sigma}(\hat{\T})-\sigma(\T))=0$,  the choice $\tau(\hat{\T})=\hat{\sigma}(\hat{\T})-\sigma(\T)$ in \eqref{disMixElasextend} leads to
\begin{align*}
\begin{aligned}
(A(\sigma-\sigma(\hat{\T})),\sigma(\hat{\T})-\sigma(\T))&
=(A(\sigma-\sigma(\hat{\T})),\sigma(\hat{\T})-\hat{\sigma}(\hat{\T})+\hat{\sigma}(\hat{\T})-\sigma(\T))\\
&=(A(\sigma-\sigma(\hat{\T})),\sigma(\hat{\T})-\hat{\sigma}(\hat{\T})).
\end{aligned}
\end{align*}
This and  Lemma~\ref{lem:intepost} show  that there exists some constant $C_0>0$ such that
\begin{equation*}
(A(\sigma-\sigma(\hat{\T})),\sigma(\hat{\T})-\sigma(\T))\leq \sqrt{C_0} \|\sigma-\sigma(\hat{\T})\|_A{\rm osc}(f,\T\setminus\hat{\T}).
\end{equation*}
The combination with Young's inequality concludes the proof.
\end{proof}

The analysis of the discrete reliability requires some $H^2(\Omega)$ conforming element. Recall the subdomains $\omega^+_{\bx_e}$  and $\omega^-_{\bx_e}$ in \eqref{KPos}.
The following $H^2$ conforming finite element space  $ S^{k+2}(\hat{\T})$ for $k\geq 3$ from  \cite{CarstensenHu2018} approximating $H^2(\Omega)$ extends the higher order Argyris finite element space   by a modification around the vertices. The   space reads
\begin{equation}\label{eq:Htwospace}
  \begin{split}
 S^{k+2}(\hat{\T}):=\big\{&\phi\in L^2(\Omega)\ \big|\forall K\in\hat{\T}  \ \phi|_K\in P_{k+2}(K;\mathbb{R});\ \phi \text{ and $\nabla \phi$ are continuous at} \\
  &\text{all vertices, $\nabla \phi$ are continuous across all interior edges, $\nabla^2 \phi$  are }\\
  &\text{continuous at each  inital vertex   $\bx\in\mathcal{V}(\T_0)$ and each boundary}\\ &\text{vertex $\bx\in\mathcal{V}(\hat{\T})\setminus\mathcal{V}_0(\hat{\T})$, $\normalvec_{e}^T\nabla^2\phi t_{e}$ and $\tangevec_{e}^T\nabla^2\phi\tangevec_{e}$} \text{ are continuous at}\\ &\text{each  internal vertex $\bx_e\in\mathcal{V}_0(\hat{\T})\setminus\mathcal{V}(\T_0)$,  $n_{ e}^T\nabla^2\phi n_{ e}$ is continuous}\\
  & \text{at $
\bx_e$ in }\omega^+_{\boldsymbol{x}} \text{ and at $\bx_e$}\text{  in $\omega^-_{\bx_e}$ for each   }\bx_e\hspace{-0.7mm}\in\hspace{-0.7mm}\mathcal{V}_0(\hat{\T})\hspace{-0.7mm}\setminus\hspace{-0.7mm}\mathcal{V}(\T_0)\big\}.\hspace{-2mm}
  \end{split}
\end{equation}
Unlike functions in the higher order  Argyris finite element space, the component $n_{ e}^T\nabla^2\phi(\hat{\T}) n_{ e}$ of the second derivatives of function $\phi(\hat{\T})\in S^{k+2}(\hat{\T})$ is not necessarily continuous at all internal vertices  $\bx_e\in\mathcal{V}_0(\hat{\T})\setminus\mathcal{V}(\T_0)$.
\begin{lemma}[discrete Helmholtz decomposition]\label{lem:disHelm}
Given any $\tau(\hat{\T})\in\widetilde{\Sigma(\hat{\T})}$ with $\div\tau(\hat{\T})=0$, there exists $\phi(\hat{\T})\in S^{k+2}(\hat{\T})$ such that $\tau(\hat{\T})={\rm \boldsymbol{CurlCurl}}\phi(\hat{\T})$.
\end{lemma}
\begin{proof}The proof is given by counting the dimensions. Let  
$$
K_{\hat{\T}}(\ddiv):=\{\tau(\hat{\T})\in\widetilde{\Sigma(\hat{\T})}\ |\ \ddiv\tau(\hat{\T})=0\}
$$
 denote the discrete divergence kernel space  with respect to $\widetilde{\Sigma(\hat{\T})}$. Note that ${\rm \boldsymbol{CurlCurl}}S^{k+2}(\hat{\T})$ is a subspace of $K_{\hat{\T}}(\ddiv)$, namely, ${\rm \boldsymbol{CurlCurl}}S^{k+2}(\hat{\T})\subseteq K_{\hat{\T}}(\ddiv)$. The dimension of   $K_{\hat{\T}}(\ddiv)$ is equal to
 the difference between the dimension of $\widetilde{\Sigma(\hat{\T})}$ and the dimension of $V(\hat{\T})$, this is
 \begin{equation*}
 \begin{split}
&3|\mathcal{V}(\hat{\T})|+|\mathcal{V}_0(\hat{\T})\setminus\mathcal{V}(\T_0)|+2(k-1)| \hat{\cE} |
 +\frac{3k(k-1)}{2}|\hat{\T}|-k(k+1)|\hat{\T}|\\
 =&3|\mathcal{V}(\hat{\T})|+|\mathcal{V}_0(\hat{\T})\setminus\mathcal{V}(\T_0)|+2(k-1)| \hat{\cE} |+\frac{k(k-5)}{2}|\hat{\T}|.
\end{split}
\end{equation*}
  The dimension of  
  ${\rm \boldsymbol{CurlCurl}}S^{k+2}(\hat{\T})$ is 
  $$
  6|\mathcal{V}(\hat{\T})|+|\mathcal{V}_0(\hat{\T})\setminus\mathcal{V}(\T_0)|+(2k-5)| \hat{\cE} |+\frac{(k-2)(k-3)}{2}| \hat{\T} |-3.$$
   Since $\Omega$ is simply-connected,  Euler's formula holds that 
   $$
   |\mathcal{V}(\hat{\T})|-| \hat{\cE} |+|\hat{\T}|=1.
   $$
     This shows that 
     $${\rm dim}{\rm \boldsymbol{CurlCurl}}S^{k+2}(\hat{\T})={\rm dim}K_{\hat{\T}}(\ddiv)$$
      and consequently implies that
      $${\rm \boldsymbol{CurlCurl}}S^{k+2}(\hat{\T})=K_{\hat{\T}}(\ddiv).$$ 
      This concludes the proof.
 \end{proof}
 \begin{lemma}[quasi-interpolation]\label{lem:inter}
 Suppose $S^{k+2}(\T)$ (resp. $S^{k+2}(\hat{\T})$) denotes   the space  in \eqref{eq:Htwospace}    with respect to $\T$ (resp. its refinement  $\hat{\T}$). There exists a quasi-interpolation $\Pi_{\T,\nabla^2}: S^{k+2}(\hat{\T})\rightarrow S^{k+2}(\T)$, which preserves the value of the function at all vertices of $\T$, for any $\phi(\hat{\T})\in S^{k+2}(\hat{\T})$,
 \begin{itemize}
 \item[(a)] for any $\bx\in\mathcal{V}(\T)$, $\Pi_{\T,\nabla^2}\phi(\hat{\T})(\bx)=\phi(\hat{\T})(\bx)$;
 \item[(b)] for any $K\in\T\setminus\hat{\T}$, $\Pi_{\T,\nabla^2}\phi(\hat{\T})|_K=\phi(\hat{\T})|_K$;
 \item[(c)] for any  $K\in\T$ with patch $\Omega(K):={\rm int}\big(\cup\{K^\prime\in\T\ |\ {\rm dist}(K,K^\prime)=0\}\big)$,
 \begin{align}\label{eq:localest}
 \sum^2_{m=0}h_K^{m-2}|\phi(\hat{\T})-\Pi_{\T,\nabla^2}\phi(\hat{\T})|_{m,K}\leq C_{\rm apx}|\phi(\hat{\T})|_{2,\Omega(K)}
 \end{align}
 with a constant $C_{\rm apx}$ that only depends on the shape regularity of $\T_0$.
 \end{itemize}
 \end{lemma}
\begin{proof}The quasi-interpolation for $S^{k+2}(\hat{\T})$  was constructed  in \cite{CarstensenHu2018}  by adopting the quasi-interpolation for the higher order Argyris finite element space in \cite{GiraultScott2002} to the extended Argyris element considered therein.  It was proven in \cite{CarstensenHu2018} that  that quasi-interpolation satisfies (b)-(c). In order to preserve the value of the function at all vertices of $\T$,  it can be done by  replacing  the combination coefficient of the basis functions associated with the vertex function value degrees of freedom by the value of functions at these vertices, namely,  it is imposed as
$$
\Pi_{\T,\nabla^2}\phi(\hat{\T})(\bx)=\phi(\hat{\T})(\bx) \text{ for any } \bx\in\mathcal{V}(\T).
$$
This concludes  the proof.
\end{proof}

\begin{lemma}\label{lem:divproest}
Let $(\sigma(\T),u(\T))\in (\widetilde{\Sigma(\T)},V(\T))$ solve  \eqref{disMixElasextend}  with respect to   $\T$, and let $(\hat{\sigma}(\hat{\T}),\inteu(\hat{\T}))\in(\widetilde{\Sigma(\hat{\T})},V(\hat{\T}))$ solve \eqref{disMixElasextendauxi}. Then it holds
\begin{equation*}
  \|\sigma(\T)-\hat{\sigma}(\hat{\T})\|^2_A\lesssim\eta^2(\T, {\T\backslash\hat{\T}}).
\end{equation*}
\end{lemma}
\begin{proof}
Let $\xi(\hat{\T}):=\sigma(\T)-\hat{\sigma}(\hat{\T})$. Since ${\rm div}\xi(\hat{\T})=0$, Lemma~\ref{lem:disHelm} implies that there exists $\phi(\hat{\T})\in S^{k+2}(\hat{\T})$ such that $\xi(\hat{\T})={\rm \boldsymbol{CurlCurl}}\phi(\hat{\T})$. Recall the quasi-interpolation $\Pi_{\T,\nabla^2} $ from Lemma~\ref{lem:inter} and let  $\psi(\hat{\T}):= \phi(\hat{\T})-\Pi_{\T,\nabla^2}\phi(\hat{\T})\in S^{k+2}(\hat{\T})$.  The choice  $\tau(\hat{\T})=\xi(\hat{\T})$ in    \eqref{disMixElasextendauxi} and the choice $\tau(\T)={\rm \boldsymbol{CurlCurl}}(\Pi_{\T,\nabla^2}\phi(\hat{\T}))$  in    \eqref{disMixElasextend} with respect to $\T$  plus (b) of Lemma~\ref{lem:inter} show
\begin{align*}
  \|\xi(\hat{\T})\|_A^2&=(A(\sigma(\T)-\hat{\sigma}(\hat{\T})),\xi(\hat{\T}))=(A\sigma(\T),{\rm \boldsymbol{CurlCurl}}\psi(\hat{\T}))\\
  &=\sum_{K\in{\T\backslash\hat{\T}}}(A\sigma(\T),{\rm \boldsymbol{CurlCurl}}\psi(\hat{\T}))_{T}.
\end{align*}
An integration by parts  leads to
\begin{equation}\label{eq:A3first}
  \begin{split}
   \|\xi(\hat{\T})\|_A^2=&\sum_{K\in{\T\backslash\hat{\T}}}\bigg(\big({\rm curlcurl}(A\sigma(\T)),\psi(\hat{\T})\big)_K+\sum_{e\in\mathcal{E}(K)}\Big(\langle A\sigma(\T)\tangevec\cdot \tangevec,{\rm \boldsymbol{Curl}}\psi(\hat{\T})\cdot\tangevec\rangle_e\\
   &-
   \langle{\rm curl}(A\sigma(\T))\cdot \tangevec,\psi(\hat{\T})\rangle_e\Big)+\sum_{e\in\mathcal{E}(K)}\langle A\sigma(\T)\tangevec\cdot \normalvec,{\rm \boldsymbol{Curl}}\psi(\hat{\T})\cdot \normalvec\rangle_e\bigg).
     \end{split}
\end{equation}
Since the compliance tensor $A$ is symmetric and continuous, $(A\sigma(\T)\tangevec)\cdot\normalvec =(A\sigma(\T)\normalvec)\cdot \tangevec=(\tangevec^T\sigma(\T)\normalvec)/(2\mu)$ is continuous across interior edge $e$.  Since   (a)  of Lemma~\ref{lem:inter} guarantees that $\psi(\hat{\T})$ vanishes at each vertex $\bx\in\mathcal{V}(\T)$, an integration by parts results in
\begin{equation*}
\begin{split}
 \sum_{K\in{\T\backslash\hat{\T}}}\sum_{e\in\mathcal{E}(K)}\langle A\sigma(\T)\tangevec\cdot \normalvec,{\rm \boldsymbol{Curl}}\psi(\hat{\T})\cdot \normalvec\rangle_e&=- \sum_{K\in{\T\backslash\hat{\T}}}\sum_{e\in\mathcal{E}(\Gamma)\cap\mathcal{E}(K)}\langle  A\sigma(\T)\tangevec_e\cdot \normalvec_e,\partial_{\tangevec_e}\psi(\hat{\T})\rangle_e\\
 &= \sum_{K\in{\T\backslash\hat{\T}}}\sum_{e\in\mathcal{E} (\Gamma)\cap\mathcal{E}(K)}\langle
 \partial_{\tangevec_e}(A\sigma(\T)\tangevec_e\cdot \normalvec_e),\psi(\hat{\T}) \rangle_e,
 \end{split}
\end{equation*}
This and  \eqref{eq:A3first} lead to
\begin{equation*}
   \|\xi(\hat{\T})\|_A^2 \leq \sum_{K\in\T\setminus\hat{\T}}\eta^2(\T,K)\bigg(h_K^{-4}\|\psi(\hat{\T})\|_{0,K}^2+\sum_{e\in\cE(K)}\left(h_K^{-1}\|{\rm\boldsymbol{Curl}}\psi(\hat{\T})\|_{0,e}^2+2h_K^{-3}\|\psi(\hat{\T})\|_{0,e}^2\right)\bigg)^{1/2}.
\end{equation*}
Trace inequalities,   the estimates in \eqref{eq:localest} for $\psi(\hat{\T})= \phi(\hat{\T})-\Pi_{\T,\nabla^2}\phi(\hat{\T})$ and $|\phi(\hat{\T})|_{2}\lesssim \|\xi(\hat{\T})\|_0$ prove
\begin{equation}\label{eq:A3sec}
  \begin{split}
   \|\xi(\hat{\T})\|_A^2 \lesssim&\sum_{K\in\T\setminus\hat{\T}}\eta(\T,K)|\phi(\hat{\T})|_{2,\Omega(K)}\lesssim\eta(\T,\T\setminus\hat{\T}) \|\xi(\hat{\T})\|_0.
  \end{split}
\end{equation}
Since ${\rm div}\xi(\hat{\T})=0$ and
$
\int_\Omega {\rm tr} \xi(\hat{\T}) \,dx=\int_\Omega{\rm tr}(\sigma(\T)-\hat{\sigma}(\hat{\T}))\,dx=0
$,   Proposition~9.1.1 of \cite{Boffi-Brezzi-Fortin2013} shows
\begin{equation*}
  \|\xi(\hat{\T})\|_{0}\lesssim\|\xi(\hat{\T})\|_A.
\end{equation*}
The combination with \eqref{eq:A3sec} concludes  the proof.
\end{proof}
\begin{theorem}[discrete reliability]\label{thm:disrelia}
There exists a positive constant $C$ such that
\begin{align*}
 \|\sigma(\hat{\T})-\sigma(\T)\|_A^2 \leq C\Big(\eta^2( \T ,\T\setminus\hat{\T})+{\rm osc}^2(f,\T\setminus\hat{\T})\Big).
\end{align*}
\end{theorem}
\begin{proof}
A triangle inequality plus Lemma~\ref{lem:intepost} and  \ref{lem:divproest} conclude the proof.
\end{proof}

With the quasi-orthogonality in Theorem~\ref{thm:quasiotho} and the discrete reliability in Theorem~\ref{thm:disrelia}, the analysis as in \cite{axioms,CKNS08,HuYu2018, HuangHuangXu2011,HuangXu}  will lead to the convergence and optimal convergence of the adaptive algorithm in Subsect.~\ref{sec:errada}.
\begin{lemma}
[estimator reduction]\label{lem:estireduc}Let $(\sigma(\T_{\ell}),u(\T_{\ell}))\in\widetilde{\Sigma(\T_\ell)}\times V(\T_\ell)$ solve \eqref{disMixElasextend} over the nested triangulations $\T_\ell$ and $\T_{\ell-1}$, respectively. Then given any positive constant $\epsilon$, there exists $\lambda:=1-2^{-1/2}<1$ and $C_\epsilon>0$ such that
\begin{align*}
\eta^2(\T_\ell)\leq (1+\epsilon)(\eta^2(\T_{\ell-1})-\lambda\eta^2(\T_{\ell-1},\mathcal{M}_{\ell-1}))+C_\epsilon\|\sigma(\T_{\ell})-\sigma(\T_{\ell-1})\|_A ^2
\end{align*}
and
\begin{align*}
{\rm osc}^2(f,\T_\ell)\leq{\rm osc}^2(f,\T_{\ell-1})-\lambda{\rm osc}^2(f,\T_\ell\setminus\T_{\ell+1}).
\end{align*}
\end{lemma}
\begin{proof}
The proof follows the same arguments as in  \cite[Corollary~4.4]{CKNS08}.
\end{proof}
\begin{theorem}Given $f\in L^2(\Omega;\mathbb{R}^2)$, let $(\sigma,u)$ denote  the exact solution   of \eqref{eqn1}, and $(\sigma(\T_{\ell}),u(\T_{\ell}))$ and $(\sigma(\T_{\ell-1}),u(\T_{\ell-1}))$ denote the discrete solutions over the nested triangulations $\cT_\ell$ and $\cT_{\ell-1}$, respectively. Then there exist positive constants $0<\alpha<1$, $\beta>0$, $\gamma>0$ such that
\begin{equation*}
  E_\ell\leq\alpha E_{\ell-1}
\end{equation*}
with
\begin{equation*}
  E_\ell=\|\sigma-\sigma(\T_{\ell})\|^2_A+\gamma\eta^2(\cT_\ell)+(\beta+\gamma){\rm osc}^2(f,\cT_\ell).
\end{equation*}
\end{theorem}
\begin{proof}Lemma~\ref{lem:estireduc},
the reliability \eqref{reliability} and the quasi-orthogonality conclude  the proof. See the details in \cite{axioms, HuYu2018,HuangXu}.
\end{proof}
For $s>0$, define the approximation class $\mathbb{A}_s$ as
\begin{equation*}
\mathbb{A}_s=\Big\{(\sigma,f):|\sigma,f|_s<\infty\text{ with }|\sigma,f|_s:=\sup_{N>0}\big(N^s\inf_{\ |\cT|-\ |\cT_0|\leq N}\inf_{\tau\in\widetilde{\Sigma(\cT)}}\|\sigma-\tau\|^2_A+{\rm osc}^2(f,\cT)\big)\Big\}.
\end{equation*}
\begin{theorem}[optimality]
Let $\mathcal{M}_\ell$ be a set of marked elements with minimal cardinality, $(\sigma,u)$ the solution of \eqref{eqn1}, and $(\cT_\ell,\widetilde{\Sigma(\T_\ell)}\times V(\T_\ell),\sigma(\T_{\ell}),u(\T_{\ell}))$ the sequence of triangulations, finite element spaces, and discrete solutions produced by the adaptive FEMs with the marking parameter $\theta$. Then it holds that
\begin{equation*}
  \|\sigma-\sigma(\T_{\ell})\|^2_A+{\rm osc}^2(f,\cT_\ell)\lesssim |\sigma,f|_s(|\cT_\ell| -|\cT_0|)^{-s}\text{ for }(\sigma,f)\in\mathbb{A}_s.
\end{equation*}
\end{theorem}
\begin{proof}
This follows from \cite{axioms,HuYu2018,HuangXu} with Lemma~\ref{lem:estireduc}, the quasi-orthogonality,  the discrete reliability  and the efficiency \eqref{efficiency} .
\end{proof}
\section{Extended stress spaces at vertex corners}\label{sec:allcorner}
This section is devoted to treat  the  boundary condition with $|\Gamma_N|>0$ and $g\not\equiv0$ in \eqref{eqn1}.
\subsection{Two dimensional case}\label{sec:corner}Recall the discrete stress space  $\Sigma(\hat{\T})$ in \eqref{Def:Odisstress}.
  To impose the general traction boundary condition $g$ on $\Gamma_N$, it requires  some approximation   $g(\hat{\T})=\alpha(\hat{\T})\normalvec|_{\Gamma_N}$ with $\alpha(\hat{\T})\in\Sigma(\hat{\T})$. Since $\alpha(\hat{\T}) $ is continuous at each boundary vertex,  $g(\hat{\T})$ may not be taken as the nodal interpolation of $g$. To see it, consider a situation depicted in   Figure~\ref{fig:cook}\subref{fig:sub2a} where  a corner vertex $\bx_c$ on the boundary of  the polygonal domain $\Omega$ is the unique intersected point of two boundary edges $e_+$ and $e_-$. Let $t_i$ and $ n_i$ denote the unit tangential and outward normal vectors of $e_i$, $i=+, -$. If the traction boundary condition $g|_{e_i}$, $i=+, -$ is consistent (continuous)
  in the sense that $( n_-^Tg|_{e_+})(\bx_c)=( n_+^Tg|_{e_-})(\bx_c)$, the usual nodal interpolation can be well defined  as follows
   \begin{equation}\label{eq4.1}
   S_{11}\varphi_{\bx_c}(\bx)\mathbb{S}_1+S_{12}\varphi_{\bx_c}(\bx)\mathbb{S}_2
   +S_{22}\varphi_{\bx_c}(\bx)\mathbb{S}_3+\tau
   \end{equation}
   with $\tau\in \Sigma(\hat{\T})$ and vanishes at $\bx_c$. Recall that $\varphi_{\bx_c}(\bx)$ is the nodal basis function associated to vertex $\bx_c$ of the scalar-valued Lagrange element of order $k$, and $\mathbb{S}_i$ is a canonical basis of the symmetric matrix space $\mathbb{S}$. The combination
   coefficients  $S_{11}$, $S_{12}$, $S_{22}$ are three components of the solution $S \in\mathbb{S}$ of the following system of equations:
      $$
      \left\{\begin{array}{l}  n_+^TS  n_+=( n_+^Tg|_{e_+})(\bx_c),\\  n_-^TS  n_-=( n_-^Tg|_{e_-})(\bx_c), \\  n_+^T S  n_-=( n_+^Tg|_{e_-})(\bx_c). \\ \end{array} \right.
      $$
   However, a general traction boundary condition can be inconsistent (discontinuous) in the sense that $ n_-^Tg|_{e_+}(\bx_c)\not=  n_+^T g|_{e_-}(\bx_c)$. Such inconsistency (discontinuity) causes an essential difficulty for imposing such a boundary condition for the   discrete stress space $\Sigma(\hat{\T})$. In particular, the usual nodal 
   interpolation as in \eqref{eq4.1} might not be defined since the combination coefficients $S_{11}$, $S_{12}$,  and $S_{22}$
    have to  satisfy the following four equations
    $$
      \left\{\begin{array}{l}  n_+^TS  n_+=( n_+^Tg|_{e_+})(\bx_c),\\  n_-^TS  n_-=( n_-^Tg|_{e_-})(\bx_c), \\  n_+^T S  n_-=( n_+^Tg|_{e_-})(\bx_c),\\
       n_-^T S  n_+=( n_-^Tg|_{e_+})(\bx_c), \\ \end{array} \right.
      $$
    while because of the inconsistency, there does not exist a solution of the above system of equations. 
     As a result, the traction boundary condition can not be exactly
  imposed even if $g|_{e_i}$, $i=+, -$,  is a polynomial of degree not bigger than $k$ for this case.
    The least-square constraint was employed in  \cite{CarstensenGunther2008} to deal with this case. The subsequent analysis  extends the treatment of non-nestedness in Subsect.~\ref {sec:extend}   to allow for an exact nodal interpolation at  corner vertices.

    The idea to overcome such a difficulty is to split  the triangle at the corner into two sub-triangles and then relax the continuity of the pure  tangential component across the common edge of these  two sub-triangles. 
 In order to accomplish this, first divide $K$ into a patch consisting of two triangles $K_+$ and $K_-$ and  $e=K_+\cap K_-$ in Figure~\ref{fig:cook}\subref{fig:sub2b}.  The pure  tangential component of discrete stress does not have to be continuous across $e$ at $\bx_c$. Therefore, one can split this degree of freedom into two separate degrees of freedom in $K_+$ and $K_-$, respectively.  In Figure~\ref{fig:cook}\subref{fig:sub2c}, the solid points represent these two separate degrees of freedom and the two circles represent the degrees of freedom of normal components across $e$. 
\begin{figure}[htbp]\centering                                                       
\subfigure[Before remedy]{                    
\begin{minipage}{3.5cm}\centering                                                          
\setlength{\unitlength}{1.4cm}
\begin{picture}(1,2)
\put(0,1){\line(1,-1){1}}\put(0.2,1.6){  $e_+$} \put(0.2,0.3){  $e_-$}
\put(1,2){\line(0,-1){2}}\put(-0.4,1){$\bx_c$}
\put(0,1){\line(1,1){1}}
\put(0,1){\circle{0.1}}
 \put(0.08,1.06){\circle{0.1}}
 \put(0.08,0.94){\circle{0.1}}
  \put(0.5,0.9){  $K$}
\end{picture}           \label{fig:sub2a}          
\end{minipage}}
\subfigure[Divide into two triangles]{                    
\begin{minipage}{3.5cm}\centering                                                          
\setlength{\unitlength}{1.4cm}
\begin{picture}(1,2)
\put(0,1){\line(1,-1){1}}
\put(0,1){\line(1,0){1}}
\put(1,2){\line(0,-1){2}}\put(-0.4,1){$\bx_c$}
\put(0,1){\line(1,1){1}}
\put(0,1){\circle{0.1}}
 \put(0.08,1.06){\circle{0.1}}\put(0.2,1.6){  $e_+$} \put(0.2,0.3){  $e_-$}
 \put(0.08,0.94){\circle{0.1}}
  \put(1,1){  $e$}  \put(0.5,1.3){  $K_+$}  \put(0.5,0.5){  $K_-$}
\end{picture}           \label{fig:sub2b}          
\end{minipage}}\subfigure[Split the tangential component]{                    
\begin{minipage}{3.5cm}\centering                                                          
\setlength{\unitlength}{1.4cm}
\begin{picture}(1,2)
\put(0,0.93){\line(1,-1){1}}
\put(1,2.07){\line(0,-1){1}}
\put(1,0.93){\line(0,-1){1}}\put(-0.4,1){$\bx_c$}
\put(0,1.07){\line(1,1){1}}
\put(0,1.07){\line(1,0){1}}
\put(0,0.93){\line(1,0){1}}
\put(0,1){\circle{0.1}}
\put(0.1,1){\circle{0.1}}
 \put(0.08,1.11){\circle*{0.1}}\put(0.2,1.6){  $e_+$} \put(0.2,0.3){  $e_-$}
 \put(0.08,0.88){\circle*{0.1}}
  \put(1,0.9){  $e$}  \put(0.5,1.3){  $K_+$}  \put(0.5,0.5){  $K_-$}
\end{picture}       \label{fig:sub2c}      
\end{minipage}   }
\subfigure[After remedy]{                    
\begin{minipage}{3cm}\centering                                                          
\setlength{\unitlength}{1.4cm}
\begin{picture}(1,2)
\put(0,1){\line(1,-1){1}}
\put(1,2){\line(0,-1){1}}
\put(1,1){\line(0,-1){1}}\put(-0.4,1){$\bx_c$}
\put(0,1){\line(1,1){1}}
\put(0,1){\line(1,0){1}}    \put(1,0.9){  $e$}  \put(0.5,1.3){  $K_+$}  \put(0.5,0.5){  $K_-$} \put(0.2,1.6){  $e_+$} \put(0.2,0.3){  $e_-$}
\put(-0.25,1.07){ $\Nwarrow$}
\put(-0.18,0.81){{$\Swarrow$}}
\end{picture}       \label{fig:sub2d}      
\end{minipage}   }\caption{Degrees  of freedom at corner  vertex $\bx_c$} 
\label{fig:cook}                                                        
\end{figure}
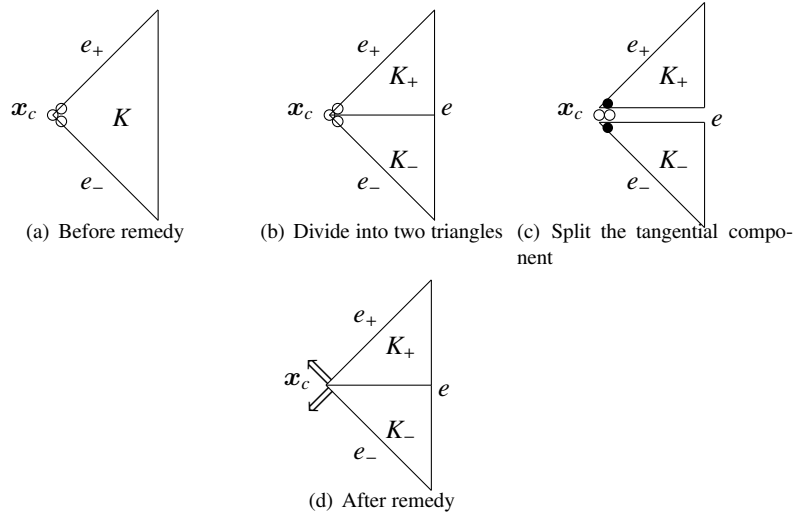
In order to define the nodal interpolation,  the four basis functions associated to vertex $\bx_c$, as depicted in Figure~\ref{fig:cook}\subref{fig:sub2d},  of the extended stress space will be presented as follows
\begin{equation}\label{eq:special}
\begin{split}
  \tau_1&=\varphi_{\bx_c}
  \begin{cases}
  \mathbb{S}_{e_+,1}^\perp+c_1\mathbb{S}_{e_+}  & \text{ on }K_+\\
  d_1\mathbb{S}_{e_-} &\text{ on }K_-
  \end{cases},\
  \tau_2=\varphi_{\bx_c}\begin{cases}
    \mathbb{S}_{e_+,2}^\perp+c_2\mathbb{S}_{e_+}  & \text{ on }K_+\\
d_2\mathbb{S}_{e_-}  &\text{ on }K_-
  \end{cases}\\
    \tau_3&=\varphi_{\bx_c}\begin{cases}
 c_3\mathbb{S}_{e_+}  & \text{ on }K_+\\
  \mathbb{S}_{e_-,1}^\perp+d_3\mathbb{S}_{e_-}  &\text{ on }K_-
  \end{cases},\,
  \tau_4=\varphi_{\bx_c}\begin{cases}
  c_4\mathbb{S}_{e_+}  & \text{ on }K_+\\
 \mathbb{S}_{e_-,2}^\perp+d_4\mathbb{S}_{e_-}  &\text{ on }K_-
  \end{cases}\\
  \end{split}
\end{equation}
with the Lagrange nodal basis function $ \varphi_{\bx_c}$  associated to   $\bx_c$ and  the other notation  defined in Section \ref{sec:Prelim}.
One can compute the constants $c_i,d_i$ for $1\leq i\leq 4$ by using the normal continuity of $\tau_i$ across $e$. For instance, given the   normal vector $n_e$ of $e$, the normal continuity of $\tau_1$ results in the following equations
 \begin{equation}\label{eq:solv}
 \big(\mathbb{S}_{e_+,1}^\perp+c_1\mathbb{S}_{e_+}\big)n_e= d_1\mathbb{S}_{e_-}n_e.
 \end{equation}
Recall the definition of the matrix $\mathbb{S}_{e_+}=t_{e_+}t_{e_+}^T$ for edge $e_+$ with   tangential vector $t_{e_+}$ and the analogy for  $\mathbb{S}_{e_-}$. Assuming $t_{e+}=(a\  b)^T$, define its perpendicular row vector by  $t_{e_+}^\perp=(b \  -\hspace{-0.1mm}a)$; $t_{e_-}^\perp$ is similar defined for $t_{e-}$. Let $D$ denote the determinant of the matrix $\begin{pmatrix}t_{e_+}\ t_{e_-}\end{pmatrix}$.
Elementary computations eventually lead to the inverse matrix
\[\begin{pmatrix}\mathbb{S}_{e_+}n_e\ -\mathbb{S}_{e_-}n_e\end{pmatrix}^{-1}=\begin{pmatrix}(t_{e_+}^Tn_e)t_{e_+}\ -(t_{e_-}^Tn_e)t_{e_-}\end{pmatrix}^{-1}=\frac{1}{D}\begin{pmatrix}
\frac{1}{t_{e_+}^Tn_e}t_{e_-}^\perp\\
\frac{1}{t_{e_-}^Tn_e}t_{e_+}^\perp
\end{pmatrix}.
\]There always exist unique   $c_1,d_1$ to \eqref{eq:solv}  unless $e_+$ is parallel to $e_-$.
\begin{remark}[more triangles]\label{rem:meshrefine}Suppose there are three triangles $K_+,K,K-$   with $e_1=K_+\cap K$ and $e_2=K_-\cap K$ around the corner vertex $\bx_c$ shown in Figure~\ref{fig:ref}.
\begin{figure}[htb]                   
\setlength\unitlength{0.6pt} \begin{picture}(300,130)(0,0)
\put(135,115){$\bx_c$}
  \put(85,36){\text{\small$K_+$}}
  \put(73,70){\text{\small$e_+$}}
    \put(107,50){\text{\small$e_1$}}
        \put(162,50){\text{\small$e_2$}}
   \put(195,70){\text{\small$e_-$}}
    \put(133,21){\text{\small$K$}}
    \put(178,36){\text{\small$K_-$}}
    \put(10,-10){\begin{picture}(-20,80)(20,0)
        \put(110.3,0){\line(-1,1){49.3}}
         \put(189.8,0){\line(1,1){49.3}}
         \put(110.2,0){\line(1,0){79.7}}
     \put(61.6,49.2){\line(5,4){88}}
       \put(150,120){\line(-2,-6){40}}
        \put(150,120){\line(2,-6){40}}
      \put(238.5,49.2){\line(-5,4){88} }
   \end{picture} }
    \end{picture}
\caption{Three triangles}\label{fig:ref}
\end{figure}
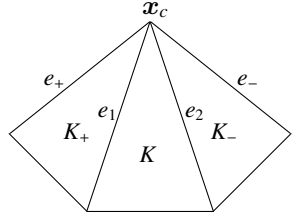In order to  relax the $C^0$ continuity at $\bx_c$,
a similar argument as in the case of two triangles leads to   five new basis functions at $\bx_c$  as follows:
\begin{align}\label{eq:specialbas}
\begin{split}
  \tau_1&=\varphi_{\bx_c}
  \begin{cases}
  \mathbb{S}_{e_+,1}^\perp+c_1\mathbb{S}_{e_+}  & \text{ on }K_+\\
  d_1\mathbb{S}_{e_2} &\text{ on }K\\
  0&\text{ on }K_-\\
  \end{cases},\
  \tau_2=\varphi_{\bx_c}\begin{cases}
    \mathbb{S}_{e_+,2}^\perp+c_2\mathbb{S}_{e_+}  & \text{ on }K_+\\
d_2\mathbb{S}_{e_2}  &\text{ on }K \\
 0&\text{ on }K_-\\
  \end{cases},\\
    \tau_3&=\varphi_{\bx_c}\begin{cases}0&\text{ on }K_+\\
 c_3\mathbb{S}_{e_1}  & \text{ on }K \\
  \mathbb{S}_{e_-,1}^\perp+d_3\mathbb{S}_{e_-}  &\text{ on }K_-
  \end{cases},\,
  \tau_4=\varphi_{\bx_c}\begin{cases}0&\text{ on }K_+\\
  c_4\mathbb{S}_{e_1}  & \text{ on }K \\
 \mathbb{S}_{e_-,2}^\perp+d_4\mathbb{S}_{e_-}  &\text{ on }K_-
  \end{cases},\\
  \tau_5&=\varphi_{\bx_c}\begin{cases}
  \mathbb{S}_{e_+}  & \text{ on }K_+\\
 c_5\mathbb{S}_{e_1,1}^\perp+d_5\mathbb{S}_{e_1,2}^\perp+g_5\mathbb{S}_{e_1}  &\text{ on }K  \\
h_5\mathbb{S}_{e_-}  & \text{ on }K_-\\
  \end{cases}.
  \end{split}
\end{align}
Using the normal continuity of $\tau_j$ across $e_1$ can compute $c_j,d_j$ for $j=1,2$. Using the normal continuity  of $\tau_j$  across $e_2$ can compute $c_j,d_j$ for $j=3,4$. Using the normal continuity across $e_1$ and $e_2$  of $\tau_5$  can compute $c_5,d_5,g_5,h_5$.

\end{remark}

\subsection{Comments for three dimensions}
This subsection extends the discrete stress space of  the mixed finte element  on tetrahedral grids \cite{Hu2015trianghigh,HuZhang2015tre} and only provides an outline. Apart from the $C^0$ continuity at vertices, the discrete stresses have some continuity across  edges.

If there are inconsistency boundary conditions   imposed on the intersection of three planes on $\Gamma_N$, it requires to deal with the degrees of freedom associated with vertices and  edges. Given  a tetrahedron $K:=\boldsymbol{x}_0\boldsymbol{x}_1\boldsymbol{x}_2\boldsymbol{x}_3$  in Figure~\ref{fig:tetrK},  suppose $\bm{x_0}$ denotes the corner vertex.   In order to relax some  continuity at vertices and edges, divide $K$ into four sub-tetrahedra with the barycentre $\bm{x}_0'$.  The goal is  to compute  the  new basis functions associated with vertex $\bm{x}_0$ and edges $\boldsymbol{x}_0\boldsymbol{x}_1$, $\boldsymbol{x}_0\boldsymbol{x}_2$, $\boldsymbol{x}_0\boldsymbol{x}_3$ as those in \eqref{eq:special}.
    Let $n_1$, $n_2$ and $n_3$ denote the outnormal of face $F_1:=\boldsymbol{x}_0\boldsymbol{x}_2\boldsymbol{x}_1$, face $F_2:=\boldsymbol{x}_0\boldsymbol{x}_1\boldsymbol{x}_3$ and face  $F_3:=\boldsymbol{x}_0\boldsymbol{x}_3\boldsymbol{x}_2$ respectively.

  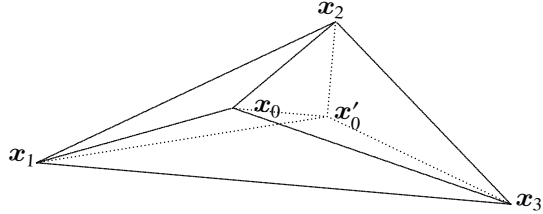
\begin{figure}[htb]
\setlength\unitlength{1.8pt}
\centering
\begin{picture}(0,50)
\put(-50,10){\begin{rotate}{-5}{ \begin{picture}(  100.,  40)(  0.,  0)
     \def\lb{\circle*{0.8}}\def\lc{\vrule width1.2pt height1.2pt}
     \def\la{\circle*{0.3}}
 \put( -6.,  0.){$\boldsymbol{x}_1$                                         }
 \put(  56.,  36.5999985){$\boldsymbol{x}_2$                                         }
 \put(  101.,  0.){$\boldsymbol{x}_3$                                         }
 \put(  43.0999985,  14.6999998){
 $\boldsymbol{x}_0$                                         }
  \put(  60,  14.6999998){
 $\boldsymbol{x}'_0$                                         }
 \multiput(   0.00,   0.00)(   0.250,   0.000){400}{\la}
 \multiput(   0.00,   0.00)(   0.234,   0.088){170}{\la}
 \multiput(   0.00,   0.00)(   0.216,   0.126){277}{\la}
 \multiput( 100.00,   0.00)(  -0.243,   0.061){247}{\la}
 \multiput(  40.00,  15.00)(   0.177,   0.177){113}{\la}
 \multiput( 100.00,   0.00)(  -0.188,   0.165){212}{\la}
 \multiput( 0.00,   0.00)(  0.80,   0.2){75}{\la}
\multiput( 100.00,   0.00)(  -0.8,   0.3){50}{\la}
 \multiput( 40.00,   15.00)( 0.8,   0){26}{\la}
 \multiput( 60.00,   35.00)( 0,   -0.8){25}{\la}
 \end{picture}}
	\end{rotate}}
\end{picture}\caption{Tetrahedron $K$}\label{fig:tetrK}
\end{figure}

  Given $1\leq m\leq 3$, note that there exist three independent symmetric matrices $\mathbb{S}_{j,m}$ in each sub-tetrahedron $K_m={\rm conv}(\bm{x}_0'\cup F_m)$ such that $\mathbb{S}_{j,m}n_m=0, 1\leq j\leq 3$ and three symmetric independent matrices $\mathbb{S}_{i,m}^\perp$ such that $\mathbb{S}_{i,m}^\perp:\mathbb{S}_{j,m}=0,1\leq i\leq 3$.
Assume the following expression
 \begin{equation*}
\begin{split}
  \tau_{i,1}&=\varphi_{\boldsymbol{x_0}}
  \begin{cases}
\mathbb{S}_{i,1}^\perp+\sum^{3}_{j=1}c_{j,1}\mathbb{S}_{j,1}& \text{ on }K_1 \\
 \sum^{3}_{j=1}d_{j,1}\mathbb{S}_{j,2} &\text{ on }K_2 \\
  \sum^{3}_{j=1}e_{j,1}\mathbb{S}_{j,3} &\text{ on }K_3\\
  \end{cases}
  \end{split}
\end{equation*}
with the   Lagrange basis function $\varphi_{\bm{x_0}}$ associated with $\bm{x_0}$ and nine constants $c_{j,1},d_{j,1},e_{j,1}$. It can be easily checked that $\tau_{i,1}n_m|_{F_m}=0$ for $m=2,3$. Using the normal continuity across the interior faces  leads to  the unique solutions $c_{j,1},d_{j,1},e_{j,1}$ as long as any two  of $F_m,1\leq m\leq 3$ do not lie in one plane. Hence $\tau_{i,1}, 1\leq i\leq 3$ are three  new basis function associated with $\bm{x_0}$. Similar arguments compute the other six functions associated with  the vertex $\bm{x_0}$.

As for any internal node $\boldsymbol{a}$ in the edge $\boldsymbol{x}_0\boldsymbol{x}_1$,  define
  \begin{equation*}
\begin{split}
  \xi_{i,1}&=\varphi_{\boldsymbol{a}}
  \begin{cases}
\mathbb{S}_{i,1}^\perp+\sum^{3}_{j=1}c_{j,1}\mathbb{S}_{j,1}& \text{ on }K_1\\
 \sum^{3}_{j=1}d_{j,1}\mathbb{S}_{j,2} &\text{ on }K_2
  \end{cases}
  \end{split}
\end{equation*}
with the  Lagrange basis function $\varphi_{\bm{a}}$ associated with $\bm{a}$ and six constants $c_{j,1},d_{j,1}$.
It should be mentioned that the normal continuity  across the interior face  $\boldsymbol{x}_1\boldsymbol{x}_0'\boldsymbol{x}_0$  only imposes three conditions. Therefore, there exist some $H({\rm div})$ bubble functions on $K_1\cup K_2$. Let $t$, $t_1$ and $t_2$ denote the tangential vector of edge $\boldsymbol{x}_0\boldsymbol{x}_1$, $\boldsymbol{x}_0\boldsymbol{x}_2$, $\boldsymbol{x}_0\boldsymbol{x}_3$, and let $n$ denote the normal vector of face $\boldsymbol{x}_1\boldsymbol{x}_0'\boldsymbol{x}_0$. The  three  $H({\rm div})$  bubble functions read
\begin{equation*}\begin{split}
b_1=\begin{cases}
 \varphi_{\boldsymbol{a}}tt^T& \text{ on }K_1\\
 0&\text{ on }K_2
  \end{cases},\
b_2=\begin{cases}
   0&\text{ on }K_1\\
 \varphi_{\boldsymbol{a}}tt^T& \text{ on }K_2
  \end{cases},\
  b_3=\begin{cases}
  \frac{\varphi_{\boldsymbol{a}}}{t_1^Tn}(t_1t^T+tt_1^T)&\text{ on }K_1\\
 \frac{\varphi_{\boldsymbol{a}}}{t_2^Tn}(t_2t^T+tt_2^T)& \text{ on }K_2
  \end{cases}.  \end{split}
\end{equation*}
Actually, $b_1 $ (resp. $b_2$)  is  already the original $H({\rm div})$ bubble function on $K_1$ (resp. $K_2$) from \cite{Hu2015trianghigh,HuZhang2015tre}. Unless $F_1$ and $F_2$ are parallel,  there always exist  the unique six constants $c_{j,1},d_{j,1}$ with the constraint of the bubble functions. Hence $\xi_{i,1}$ for $ 1\leq i\leq 3$ are three  new basis function associated with edge $\boldsymbol{x}_0\boldsymbol{x}_1$. 
\section{Numerics}\label{Numerics}
This section provides five examples to  compare the  conforming mixed element of $k=3$  in \cite{HuZhang2014a}  introduced in Subsect.~\ref{sec:orignalMFEM}  before and after relaxing the $C^0$ vertex continuity of stress spaces. Example~5.2-5.4  show the results by relaxing the $C^0$ continuity at corner vertices in Subsect.~\ref{sec:corner} on uniform and adaptive meshes. The last example demonstrates convergence rates of  the adaptive  mixed finite element with nested stress spaces  in Subsect.~\ref{sec:extend}.
\subsection{Interface problem}\label{sec:Dislinecontinu}
Let $\boldsymbol{x}=(x_1,x_2)^T$. Consider a piecewise constant stress
 $\sigma=\begin{pmatrix}
                                                                             \sigma_{11} & 0 \\
                                                                        0 &  0 \\
                                                                           \end{pmatrix}
$ with a discontinuous  pure tangential component $\sigma_{11}$ across $x_2=0.5$ depicted in Figure~ \ref{fig:interface}\subref{fig:sub1a}.  Suppose that the intersection of any edge  of  $\hat{\T}$ with $x_2=0.5$ is an empty set, a vertex or  the edge itself.  Recall the stress space  $\Sigma(\hat{\T})$ in \eqref{Def:Odisstress}.
Since any matrix-valued function in  $\Sigma(\hat{\T})$ is $C^0$ continuous at each vertex of $\hat{\T}$. The mixed FEM \eqref{disMixElas} cannot achieve  the exact stress  due to  $\sigma\not\in\Sigma(\hat{\T})$.

For   vertex $\bm{a}$  on $x_2=0.5$ shown in Figure~\ref{fig:interface}, employ  the treatment in Subsect.~\ref{sec:extend}  to split the degree of freedom with respect to the pure tangential component along $x_2=0.5$  into two degrees of freedom shared by the upper plane and lower plane, respectively.  The two solid points in Figure~\ref{fig:interface} represent the two new separate degrees of freedom; the two circles  represent  the degrees of freedom with respect to  normal components. Deal with all vertices on $x_2=0.5$ analogously. An extended stress space $\widetilde{\Sigma(\hat{\T})}$ is then constructed by enriching  $\Sigma(\hat{\T})$ with all the new  functions whose pure tangential component  are discontinuous  across $x_2=0.5$.
\begin{figure}[!ht]\centering                                                  
\subfigure[Before remedy]{                    
\begin{minipage}{6cm}\centering                                                          
\setlength{\unitlength}{1.8cm}
\begin{picture}(2,2)
\put(0,0){\line(1,0){2}}
\put(0,0){\line(0,1){2}}
\put(2,0){\line(0,1){2}}
\put(0,2){\line(1,0){2}}
\put(0,1){\line(1,0){2}}
\put(1,1){\line(1,1){1}}
\put(1,1){\line(-1,1){1}}
\put(1,1){\line(-1,-1){1}}
\put(1,1){\line(1,-1){1}}
 \put(-0.8,1){$ x_2=0.5$}
  \put(0.9,1.1){ $\boldsymbol{a}$}
   \put(0.6,0.5){ $\sigma_{11}=1$}
    \put(0.5,1.5){ $\sigma_{11}=10$}
\put(1.,1.04){\circle{0.09}}
\put(0.95,0.96){\circle{0.08}}
\put(1.05,0.96){\circle{0.08}}
\end{picture}    \label{fig:sub1a}          
\end{minipage}}\subfigure[After remedy]{                    
\begin{minipage}{7cm}\centering                                                          
\setlength{\unitlength}{1.8cm}
\begin{picture}(2,2)
\put(0,0.0){\line(1,0){2}}
\put(0,0.0){\line(0,1){1}}
\put(0,1.1){\line(0,1){1}}
\put(2,-0){\line(0,1){1}}
\put(2,1.1){\line(0,1){1}}
\put(0,2.1){\line(1,0){2}}
\put(0,1.1){\line(1,0){2}}
\put(0,1){\line(1,0){2}}
\put(1,1.1){\line(1,1){1}}
\put(1,1.1){\line(-1,1){1}}
\put(1,1){\line(-1,-1){1}}
\put(1,1){\line(1,-1){1}}
 \put(-0.80,1){$ x_2=0.5$}
  \put(0.9,1.2){ $\boldsymbol{a}$}
   \put(0.6,0.5){ $\sigma_{11}=1$}
    \put(0.5,1.6){ $\sigma_{11}=10$}
\put(1,0.97){\circle*{0.09}}
\put(1,1.13){\circle*{0.09}}
\put(0.95,1.05){\circle{0.08}}
\put(1.05,1.05){\circle{0.08}}
\end{picture}       \label{fig:sub1b}      
\end{minipage}   }\caption{Degrees  of freedom at vertex $a$} \label{fig:interface}
\end{figure}
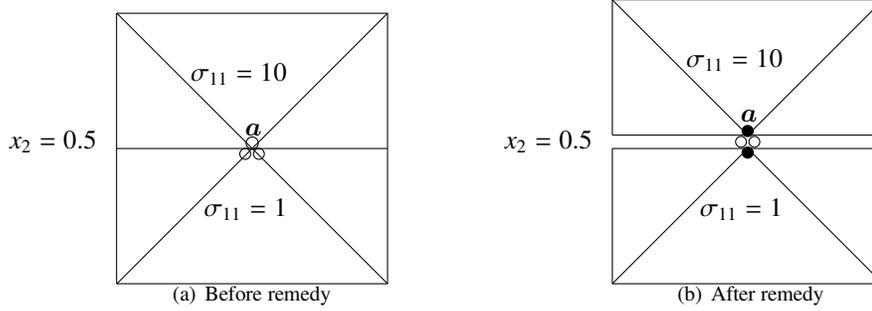

Consider the interpolation of $\sigma$ in $\widetilde{\Sigma(\hat{\T})}$. Let the interpolation $I_{\hat{\T}}\sigma_{11}$ equal to $1$ at vertex ${\boldsymbol{a}}$ in each element below $x_2=0.5$, and  $10$  in each element above. Thus $I_{\hat{\T}}\sigma_{11}=\sigma_{11}$. Since the computing error is less than the interpolation error, the mixed FEM with  the extended stress space $\widetilde{\Sigma(\hat{\T})}$ computes the exact stress.
\subsection{A benchmark problem over an L-shaped domain  with treatment of corners on uniform  meshes}Consider  the model problem on the rotated L-shaped domain $\Omega\subset\mathbb{R}^2$ as depicted in Figure \ref{fig:Lshapemesh}.  The exact solution reads in polar coordinates
\begin{equation*}
\begin{split}
  &u_r(r,\phi)=\frac{r^\alpha}{2\mu}(-(\alpha+1)\cos((\alpha+1)\phi)+(C_2-\alpha-1)C_1\cos((\alpha-1)\phi)),\\
  & u_\phi(r,\phi)=\frac{r^\alpha}{2\mu}((\alpha+1)\sin((\alpha+1)\phi)+(C_2+\alpha-1)C_1\sin((\alpha-1)\phi)).
  \end{split}
\end{equation*}
The constants are $C_1:=-\cos((\alpha+1)\omega)/\cos((\alpha-1)\omega)$ and $C_2:=2(\lambda+2\mu)/(\lambda+\mu)$, where $\alpha=0.544483736782$ is the positive solution of $\alpha\sin(2\omega)+\sin(2\omega\alpha)=0$ for $\omega=3\pi/4$ and with Lam\'{e} parameter $\lambda$ and $\mu$ according to the elasticity modulus is $E=10^5$ and the Poisson's ratio $\nu=0.499$. The volume force and the traction boundary data vanish, and the Dirichlet boundary conditions are taken from the exact solution. The exact solution exhibits a strong singularity at the origin $\bx_c$.

The sequence of meshes is generated uniformly by the initial mesh from Figure~\ref{fig:Lshapemesh}.  It  is unnecessary to deal with the degrees of freedom at corner vertices due to the zero traction boundary data. However, considering the singularity  of the exact solution at $\bx_c $, we would like to investigate the performance of the mixed FEM in \cite{HuZhang2014a} after  relaxing the $C^0$ continuity at $\bx_c$ as in  Subsect.~\ref{sec:corner}.  Table \ref{Tab:lshape} demonstrates  that after the remedy,  the errors of $\|\sigma-\sigma(\hat{\T})\|_A$ and $\|u-u(\hat{\T})\|_0$ are largely reduced. This is because the relaxation of the continuity at the origin $\bx_c$ adds several more degrees of freedom  at  $\bx_c$.
\begin{figure}[h!]\label{fig:Lshapemesh}
\includegraphics[width=5in]{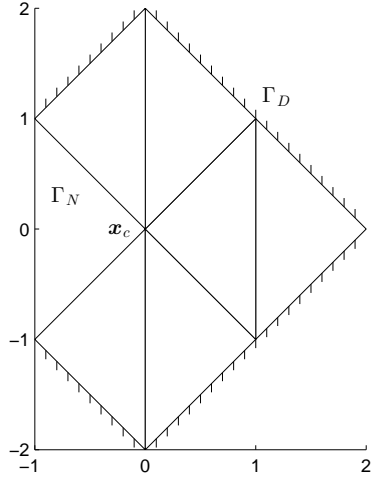}
\caption{Initial mesh for the L-shaped domain}
\end{figure}
\begin{table}[h!]\caption{ A benchmark problem on L-shaped domain on uniform meshes}
\begin{tabular}{|c|c|c|c|}
\hline
\multicolumn{4}{|c|}{$\|\sigma-\sigma(\hat{\T})\|_A$}\\
\hline
Before remedy& Order &After remedy&Order\\
\hline
7.5143E-03	&		&	3.1046E-03	&				\\
5.5899E-03	&	0.4268	&	2.1945E-03	&	0.5005\\
3.9400E-03	&	0.5046 	&	1.5147E-03	&	0.5348 	\\
2.7379E-03	&	0.5251 	&	1.0416E-03	&	0.5402	\\
1.8894E-03	&	0.5351 	&	7.1517E-04	&	0.5424\\

\hline
\multicolumn{4}{|c|}{$\|u-u(\hat{\T})\|_0$}\\
\hline
7.2407E-06	&	 	&	2.1190E-06	&			\\
3.8230E-06	&	 0.9214	&	8.9234E-07	&1.2477 	  	\\
1.8422E-06	&	 1.0533&	3.6866E-07	&	 1.2753\\
8.8088E-07	&	 1.0644&	1.5774E-07	&	 1.2247 \\
4.1866E-07	&	1.0732&	6.9708E-08	&	1.1782	\\
  \hline
\end{tabular}\label{Tab:lshape}
\end{table}
\subsection{Cook's membrane problem  with treatment of corners on uniform  meshes}\label{sec:cook}The underlying domain is a quadrilateral with vertices at $(0,0), (48,44),(48,60)$ and $(0,44)$. It is fixed $(u=0)$ at the left edge of the boundary $(x_1=0)$ while a uniform traction force
pointing upwards $(\sigma n=(0,1)^T)$ is applied at the right edge $(x_1=48)$. At the remaining part of the boundary it is kept in equilibrium $(\sigma n=0)$, see Figure \ref{cookmesh}. The elasticity modulus is $E=10^5$ and the Poisson's ratio $\nu=0.499$. At the corner nodes $\bx_{c,1}$ and $\bx_{c,2}$, there exist inconsistent traction  boundary conditions.

This example relaxes the $C^0$ continuity at $\bx_{c,1}$ and $\bx_{c,2}$. Since the exact solution is unknown, the fine grid approximation is computed by the standard continuous $P_5$-FEM on twice  uniform refinements of the grid for the last level.  Since the inconsistency of the boundary conditions at $\bx_{c,1}$ and $\bx_{c,2}$,  the least squares method is employed before remedy to decide the values of stress at the two right corners.  Table \ref{Tab:cook} shows that after the remedy, the errors are reduced, especially on  coarse meshes. This happens because  the error of   inexact traction boundary conditions dominates on  coarse meshes.
\begin{figure}[h!]\label{cookmesh}
\includegraphics[width=4in]{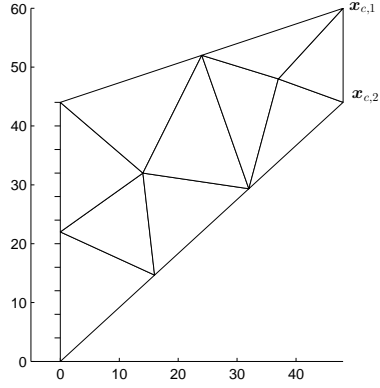}
\caption{Initial mesh for Cook's membrane problem}
\end{figure}
\begin{table}[h!]\caption{Cook's membrane problem on uniform meshes}
\begin{tabular}{|c|c|c|c|}
  \hline
\multicolumn{4}{|c|}{$\|\sigma-\sigma(\hat{\T})\|_A$}\\
\hline
Before remedy& Order&After remedy& Order\\
\hline							
0.0488	&		&	0.0341 	&			\\
0.0299	&	0.7048 	&	0.0241 	&	0.5001 		\\
0.0190	&	0.6556 	&	0.0167	&	0.5286	\\
0.0121 	&	0.6470 	&	0.0113 	&	0.5685  	\\
\hline
\multicolumn{4}{|c|}{$\|u-u(\hat{\T})\|_0$}\\
\hline
6.2925E-03	&		&	2.1216-03	&		\\
3.0423E-03	&	1.0485 	&	1.0712E-03	&	0.9859 		\\
1.4886E-03	&	1.0311 	&	5.3616E-04	&	0.9985  	\\
7.2422E-04	&	1.0395 	&	2.7651E-04	&	0.9553 	 	\\
  \hline
\end{tabular}\label{Tab:cook}
\end{table}
\subsection{Cook's membrane problem  with treatment of corners on adaptive meshes}
In the case of $|\Gamma_N|\neq 0$ and $u_D\not\equiv0$ on $\Gamma_D$,  the error estimator  $\eta^2(\hat{\T})$ modifies $ \mathcal{J}_{e,1}$ and $ \mathcal{J}_{e,2}$ in \eqref{eq:estimator} from \cite[Sec.~4]{ChenHuHuangMan2017} as follows
$$
\begin{array}{ll}
 \disp \mathcal{J}_{e,1}:=&\left\{\begin{array}{ll}
 \disp\Big[(A\sigma(\hat{\T}))\tangevec_e\cdot\tangevec_e\Big]_e&{\rm if~} e\in\hat{\mathcal{E}}(\Omega),\\
 \disp\Big((A\sigma(\hat{\T}))\tangevec_e\cdot\tangevec_e-\partial_{\tangevec_e}(u_D\cdot\tangevec_e) \Big)\big|_e\quad\quad\quad\quad\quad\ & {\rm if~} e\in\hat{\mathcal{E}}(\Gamma_D),
 \end{array}\right. \\
 \\
 \disp \mathcal{J}_{e,2}:=&\left\{\begin{array}{lr}
\disp \Big[{\rm curl}(A\sigma(\hat{\T}))\cdot \tangevec_e\Big]_e& {\rm if~} e\in\hat{\mathcal{E}}(\Omega),\hspace{2mm}\\
\disp \Big({\rm curl}(A\sigma(\hat{\T}))\cdot \tangevec_e +\partial _{\tangevec_e\tangevec_e}(u_D\cdot\normalvec_e)-\partial _{\tangevec_e}\big((A\sigma(\hat{\T}))\tangevec_e\cdot \normalvec_e\big)\Big)\big|_e&\quad {\rm if~} e\in\hat{\mathcal{E}}(\Gamma_D).\hspace{0.5mm}
 \end{array}\right.
 \end{array}
$$
The reliability holds \begin{equation*}
  \|\sigma-\sigma(\hat{\T})\|^2_A\leq \widetilde{C_{Rel}} \big(\eta^2( \hat{\T})+{\rm osc}^2(f,\hat{\T})+{\rm osc}^2(g,\hat{\cE}(\Gamma_N))\big)
\end{equation*}
with the data oscillation   ${\rm osc}^2(g, \hat{\cE} (\Gamma_N)):=\sum_{e\in\hat{\cE}(\Gamma_N)}h_e\|g-g(\hat{\T})\|^2_{0,e}$.

Define the total estimator $\eta:=\big(\eta^2( \hat{\T})+{\rm osc}^2(f,\hat{\T})+{\rm osc}^2(g,\hat{\cE}(\Gamma_N))\big)^{1/2}$. The discrete stress and  the total estimator  by the mixed FEM \eqref{disMixElas} before treatment  is denoted by $\sigma(\hat{\T})^O$ and $\eta^O$. The corresponding results after the relaxation of the $C^0$ continuity at $\bx_{c,1}$ and $\bx_{c,2}$ are  denoted by $\sigma(\hat{\T})^M$ and $\eta^M$.  The approximations by  the standard continuous $P_5$ element on the very fine mesh are computed as the reference solutions. The results are shown in Figure~\ref{cook-compare}. Since   the error of   inexact boundary conditions  dominates
at the initial steps, the result without treatment is worse than that with treatment. After several refinements, the results are almost the same since  the error of   inexact boundary conditions  has deduced greatly.
\begin{figure}[h!]\label{cook-compare}
\includegraphics[width=8cm]{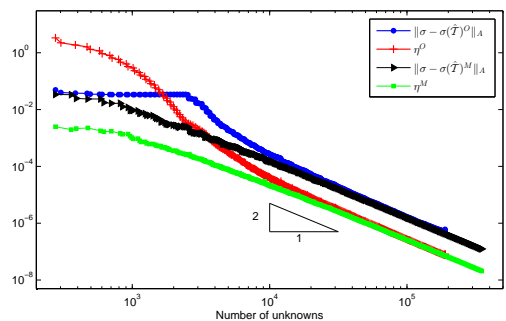}
\caption{Errors $\|\sigma-\sigma(\hat{\T})\|_A$ and $\eta$ vs $\#$dofs for the Cook's membrane problem}
\end{figure}
\subsection{Comparison of adaptivity algorithms for two mixed elements}This example   compares the mixed element of $k=3 $ in  \cite{HuZhang2014a}  and  the extended mixed element in Subsect.~\ref{sec:extend} for the benchmark problem over an L-shaped domain. The results are shown in Figure~\ref{LshapeAdap}. Let $\sigma(\hat{\T})^O$ and $\eta^O$ denote the discrete stress and the total estimator  of the former element , and the corresponding results of the extended element by $\sigma(\hat{\T})^E$ and $\eta^E$.  Figure~\ref{LshapeAdap} presents the convergence history plot and  illustrates that  there is not much difference between  these two elements.
\begin{figure}[h!]\label{LshapeAdap}
\includegraphics[width=8cm]{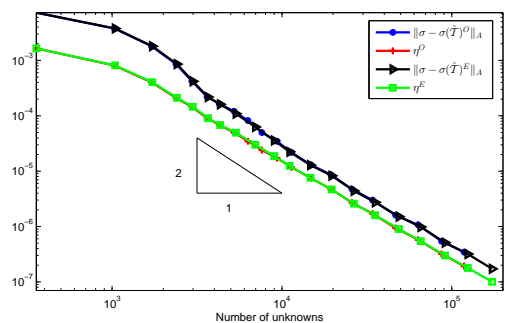}
\caption{Errors $\|\sigma-\sigma(\hat{\T})\|_A$ and $\eta$ vs $\#$dofs for the problem on an L-shaped domain}
\end{figure}
\bibliographystyle{siamplain}
\bibliography{lit_mit_doi}
\end{document}